\documentclass[12pt]{amsart} 
\usepackage{amsfonts,amsmath,amsthm,amsfonts,amscd,amssymb,amsmath,latexsym,multicol,euscript}
\usepackage[normalem]{ulem}
\usepackage{epsfig,graphicx,flafter,caption}
\usepackage{enumerate}
\usepackage[active]{srcltx}
\usepackage{hyperref,url}
\usepackage{color}
\usepackage{asymptote,asyfig}
\usepackage{media9}
\usepackage[small,nohug,UglyObsolete]{diagrams}
\diagramstyle[labelstyle=\scriptstyle]
\diagramstyle[noPostScript]
\usepackage{algorithm}
\usepackage[noend]{algpseudocode}
\usepackage{xr}
\usepackage{polynom}
\usepackage{csquotes}
\usepackage{epigraph}
\makeatletter

\def\jobis#1{FF\fi
  \def\preedicate{#1}%
  \edef\preedicate{\expandafter\strip@prefix\meaning\preedicate}%
  \edef\job{\jobname}%
  \ifx\job\preedicate
}

\makeatother

\if\jobis{proposal}%
 \def\try{subsection}%
\else
  \def\try{section}%
\fi

\theoremstyle{plain}
\newtheorem{theorem}{Theorem}[\try]
\newtheorem{corollary}[theorem]{Corollary}
\newtheorem{lemma}[theorem]{Lemma}

\newtheorem{proposition}[theorem]{Proposition}
\newtheorem{definition-lemma}[theorem]{Definition-Lemma}
\newtheorem{definition-proposition}[theorem]{Definition-Proposition}
\newtheorem{definition-theorem}[theorem]{Definition-Theorem}

\newtheorem{definition}[theorem]{Definition}

\def\lfomitlist#1.#2.#3.#4.{{#1}_0,{#1}_1 #2 \dots #2\hat{{#1}_{#4}} #2\dots #2 {#1}_{#3}}

\def\alist#1.#2.#3.{{#1}_1 #2 {#1}_2 #2\dots #2 {#1}_{#3}}
\def\zlist#1.#2.#3.{#1_0 #2 #1_1 #2\dots #2 #1_{#3}}
\def\ltomitlist#1.#2.#3.{{#1}_0,{#1}_1 #2 \dots #2\hat {{#1}_i} #2\dots #2 {#1}_{#3}}
\def\lomitlist#1.#2.#3.{{#1}_0 #2 {#1}_1 #2 \dots #2 \hat {{#1}_i} #2 \dots #2 {#1}_{#3}}
\def\lmap#1.#2.#3.{#1 \overset{#2}{\longrightarrow} #3}
\def\mes#1.#2.#3.{#1 \longrightarrow #2 \longrightarrow #3}
\def\ses#1.#2.#3.{0\longrightarrow #1 \longrightarrow #2 \longrightarrow #3 \longrightarrow 0}
\def\les#1.#2.#3.{0\longrightarrow #1 \longrightarrow #2 \longrightarrow #3}
\def\res#1.#2.#3.{#1 \longrightarrow #2 \longrightarrow #3\longrightarrow 0}
\def\Hi#1.#2.#3.{\text {Hilb}^{#1}_{#2}(#3)}
\def\ten#1.#2.#3.{#1\underset {#2}{\otimes} #3}

\def\mderiv#1.#2.#3.{\frac {d^{#3} #1}{d #2^{#3}}}
\def\mfderiv#1.#2.#3.{\frac {\partial^{#3} #1}{\partial #2}}
\def\ggr#1.#2.#3.{\mathbb{G}_{#1}(#2,#3)}

\def\llist#1.#2.{{#1}_1,{#1}_2,\dots,{#1}_{#2}}
\def\ulist#1.#2.{{#1}^1,{#1}^2,\dots,{#1}^{#2}}
\def\lomitlist#1.#2.{{#1}_1,{#1}_2,\dots,\hat {{#1}_i}, \dots, {#1}_{#2}}
\def\lomitlistz#1.#2.{{#1}_0,{#1}_1,\dots,\hat {{#1}_i}, \dots, {#1}_{#2}}
\def\loc#1.#2.{\Cal O_{#1,#2}}
\def\fderiv#1.#2.{\frac {\partial #1}{\partial #2}}
\def\deriv#1.#2.{\frac {d #1}{d #2}}
\def\map#1.#2.{#1 \longrightarrow #2}
\def\rmap#1.#2.{#1 \dasharrow #2}
\def\emb#1.#2.{#1 \hookrightarrow #2}
\def\non#1.#2.{\text {Spec }#1[\epsilon]/(\epsilon)^{#2}}
\def\Hi#1.#2.{\text {Hilb}^{#1}(#2)}
\def\sym#1.#2.{\operatorname {Sym}^{#1}(#2)}
\def\Hb#1.#2.{\text {Hilb}_{#1}(#2)}
\def\Hm#1.#2.{\Hom_{#1}(#2)}
\def\prd#1.#2.{{#1}_1\cdot {#1}_2\cdots {#1}_{#2}}
\def\Bl #1.#2.{\operatorname {Bl}_{#1}#2}
\def\pl #1.#2.{#1^{\otimes #2}}
\def\mgn#1.#2.{\overline {M}_{#1,#2}}
\def\ialist#1.#2.{{#1}_1 #2 {#1}_2 #2 {#1}_3 #2\dots}
\def\pair#1.#2.{\langle #1, #2\rangle}
\def\gproj#1.#2.{\mathbb{P}_{#1}(#2)}
\def\gpr #1.#2.{\mathbb{P}^{#1}_{#2}}
\def\gaf #1.#2.{\mathbb{A}^{#1}_{#2}}
\def\vandermonde#1.#2.{\left|
\begin{matrix}
1 & 1 & 1 & \dots & 1\\
{#1}_1 & {#1}_2 & {#1}_3 & \dots & {#1}_{#2}\\
{#1}_1^2 & {#1}_2^2 & {#1}_3^2 & \dots & {#1}_{#2}^2\\
\vdots & \vdots & \vdots & \ddots & \vdots\\
{#1}_1^{#2-1} & {#1}_2^{#2-1} & {#1}_2^{#2-1} & \dots & {#1}_{#2}^{#2-1}\\
\end{matrix}
\right|
}
\def\vandermondet#1.#2.{\left|
\begin{matrix}
1 & {#1}_1   & {#1}_1^2 & \dots & {#1}_1^{#2-1}\\
1 & {#1}_2   & {#1}_2^2 & \dots & {#1}_2^{#2-1}\\
1 & {#1}_3   & {#1}_3^2 & \dots & {#1}_3^{#2-1}\\
\vdots & \vdots & \vdots & \ddots & \vdots\\
1 & {#1}_{#2}& {#1}_{#2}^2 & \dots & {#1}_{#2}^{#2-1}\\
\end{matrix}
\right|
}
\def\gr#1.#2.{\mathbb{G}(#1,#2)}
\def\bdd#1.#2.{{#1}_{\rfdown #2.}}

\def\ideal#1.{I_{#1}}
\def\ring#1.{\mathcal {O}_{#1}}
\def\fring#1.{\hat{\mathcal {O}}_{#1}}
\def\aring#1.{{\mathcal {O}}_{#1}^{\text{an}}}
\def\proj#1.{\mathbb {P}(#1)}
\def\pr #1.{\mathbb {P}^{#1}}
\def\dpr #1.{\hat{\mathbb {P}}^{#1}}
\def\af #1.{\mathbb{A}^{#1}}
\def\Hz #1.{\mathbb{F}_{#1}}
\def\Hbz #1.{\overline{\mathbb {F}}_{#1}}
\def\fb#1.{\underset {#1} {\times}}
\def\rest#1.{\underset {\ \ring #1.} \to \otimes}
\def\au#1.{\operatorname {Aut}\,(#1)}
\def\deg#1.{\operatorname {deg } (#1)}
\def\pic#1.{\operatorname {Pic}\,(#1)}
\def\pico#1.{\operatorname{Pic}^0(#1)}
\def\picg#1.{\operatorname {Pic}^G(#1)}
\def\n1#1.{\operatorname{N^1}(#1)}
\def\ner#1.{\operatorname{NS}(#1)}
\def\rdown#1.{\llcorner#1\lrcorner}
\def\rfdown#1.{\lfloor{#1}\rfloor}
\def\rup#1.{\ulcorner{#1}\urcorner}
\def\rfup#1.{\lceil{#1}\rceil}
\def\sship#1.{\langle{#1}\rangle}

\def\bp#1.{#1^{{}\leq 1}}
\def\rcup#1.{\lceil{#1}\rceil}
\def\cone#1.{\operatorname {NE}(#1)}
\def\mone#1.{\operatorname {NM}(#1)}
\def\none#1.{\operatorname {NF}(#1)}
\def\ccone#1.{\overline{\operatorname {NE}}(#1)}
\def\cmone#1.{\overline{\operatorname {NM}}(#1)}
\def\cnone#1.{\overline{\operatorname {NF}}(#1)}
\def\cbig#1.{\overline{\operatorname {B}(#1)}}
\def\coef#1.{\frac{(#1-1)}{#1}}
\def\vit#1.{D_{\langle #1 \rangle}}
\def\mm#1.{\overline {M}_{0,#1}}
\def\Hone#1.{H^1(#1,{\ring #1.})}
\def\ac#1.{\overline {\mathbb F}_{#1}}

\def\adj#1.{\frac {#1-1}{#1}}
\def\spn#1.{\overline{#1}}
\def\pek#1.#2.{\Cal P^{#1}(#2)}
\def\plk#1.#2.{\Cal P^{\leq #1}(#2)}
\def\ev#1.{\operatorname{ev_{#1}}}
\def\ilist#1.{{#1}_1,{#1}_2,\dotsc}
\def\bminv#1.{(\nu_1,s_1;\nu_2,s_2;\dots ;\nu_{#1},s_{#1};\nu_{r+1})}
\def\zinv#1.{(\nu_1,s_1;\nu_2,s_2;\dots ;\nu_{#1},s_{#1};0)}
\def\iinv#1.{(\nu_1,s_1;\nu_2,s_2;\dots ;\nu_{#1},s_{#1};\infty)}
\def\scr#1.{\mathbf{\EuScript{#1}}}
\def\mg#1.{\overline {M}_{#1}}
\def\inter#1.{\underset #1{\cdot}}
\def\cate#1.{\text{(\underline{#1})}}
\def\dls#1.{\overrightarrow{#1}}
\def\id#1.{\text{id}_{#1}}

\def\Hom{\operatorname{Hom}}

\def\Proj{\operatorname{Proj}}

\def\dim{\operatorname{dim}}

\def\deg{\operatorname{deg}}

\def\Pic{\operatorname{Pic}}

\def\im{\operatorname{Im}}

\def\mult{\operatorname{mult}}

\def\rest{\operatorname{res}}

\def\vol{\operatorname{vol}}

\def\C`har{\operatorname{char}}

\def\Isom{\operatorname{Isom}}

\def\C{\mathbb C}

\def\e{\Cal E}

\def\e1{E_1}
\def\e2{E_2}

\def\mslc{\mathcal{M}^{\text{slc}}}
\def\mapdown#1{\big\downarrow\rlap{$\vcenter
{\hbox{$\scriptstyle#1$}}$}}

\def\mapse#1{
{\vcenter{\hbox{$\mathop{\smash{\raise1pt\hbox{$\diagdown$}\!\lower7pt
\hbox{$\searrow$}}\vphantom{p}}\limits_{#1}\vphantom{\mapdown{}}$}}}}

\def\VR#1.{height#1pt&\omit&&\omit&&\omit&&\omit&&\omit&\cr}

\def\VRT#1.{height#1pt&\omit&&\omit&\cr}

\makeatletter
\renewcommand*\env@matrix[1][*\c@MaxMatrixCols c]{%
  \hskip -\arraycolsep
  \let\@ifnextchar\new@ifnextchar
  \array{#1}}
\makeatother

\begin{document}
\title{Boundedness of moduli of varieties of general type}
\date{\today}
\author{Christopher D. Hacon}
\address{Department of Mathematics \\
University of Utah\\
155 South 1400 East\\
JWB 233\\
Salt Lake City, UT 84112, USA}
\email{hacon@math.utah.edu}
\author{James M\textsuperscript{c}Kernan}
\address{Department of Mathematics\\
University of California, San Diego\\
9500 Gilman Drive \# 0112\\
La Jolla, CA  92093-0112, USA}
\email{jmckernan@math.ucsd.edu}
\author{Chenyang Xu}
\address{Beijing International Center of Mathematics Research\\ 
5 Yiheyuan Road, Haidian District\\ 
Beijing 100871, China}
\email{cyxu@math.pku.edu.cn}

\thanks{The first author was partially supported by DMS-1300750, DMS-1265285 and a grant
  from the Simons foundation, the second author was partially supported by NSF research
  grant no: 0701101, no: 1200656 and no: 1265263 and this research was partially funded by
  the Simons foundation and by the Mathematische Forschungsinstitut Oberwolfach and the
  third author was partially supported by \lq\lq The Recruitment Program of Global
  Experts\rq\rq\ grant from China.  Part of this work was completed whilst the second and
  third authors were visiting the Freiburg Institute of Advanced Studies and they would
  like to thank Stefan Kebekus and the Institute for providing such a congenial place to
  work.  We are grateful to J\'anos Koll\'ar and Mihai P\u aun for many useful comments
  and suggestions and to the referee for a careful reading of this paper and some very
  helpful comments.}

\begin{abstract} We show that the family of semi log canonical pairs with ample log
canonical class and with fixed volume is bounded.
\end{abstract}

\maketitle

\tableofcontents

\section{Introduction}

The aim of this paper is to show that the moduli functor of semi log canonical stable
pairs is bounded:
\begin{theorem}\label{t_modulifinite} Fix an integer $n$, a positive rational number 
$d$ and a set $I\subset [0,1]$ which satisfies the DCC.

Then the set $\mathfrak{F}_{\text{slc}}(n,d,I)$ of all log pairs $(X,\Delta)$ such that
\begin{enumerate} 
\item $X$ is projective of dimension $n$, 
\item $(X,\Delta)$ is semi log canonical, 
\item the coefficients of $\Delta$ belong to $I$, 
\item $K_X+\Delta$ is an ample $\mathbb{Q}$-divisor, and 
\item $(K_X+\Delta)^n=d$,
\end{enumerate} 
is bounded.  

In particular there is a finite set $I_0$ such that
$\mathfrak{F}_{\text{slc}}(n,d,I)=\mathfrak{F}_{\text{slc}}(n,d,I_0)$.
\end{theorem}

The main new technical result we need to prove \eqref{t_modulifinite} is to show that
abundance behaves well in families:
\begin{theorem}\label{t_one-smooth} Suppose that $(X,\Delta)$ is a log pair where
the coefficients of $\Delta$ belong to $(0,1]\cap \mathbb{Q}$.  Let $\pi\colon\map X.U.$
be a projective morphism to a smooth variety $U$.  Suppose that $(X,\Delta)$ is log smooth
over $U$.

If there is a closed point $0\in U$ such that the fibre $(X_0,\Delta_0)$ has a good
minimal model then $(X,\Delta)$ has a good minimal model over $U$ and every fibre has a
good minimal model.
\end{theorem}

\begin{corollary}\label{c_dense} Let $(X,\Delta)$ be a log pair where $\Delta$ is a 
$\mathbb{Q}$-divisor and let $\map X.U.$ be a flat projective morphism to a variety $U$.
Suppose that $U$ is smooth and the support of $\Delta$ contain neither a component of any
fibre nor a codimension one component of the singular locus of a fibre.

Then the subset $U_0\subset U$ of points $u\in U$ such that the fibre $(X_u,\Delta_u)$ is
divisorially log terminal and has a good minimal model is constructible.  
\end{corollary}

\begin{corollary}\label{c_deformation} Let $\pi\colon\map X.U.$ be a projective morphism 
to a smooth variety $U$ and let $(X,\Delta)$ be log smooth over $U$.  Suppose that the
coefficients of $\Delta$ belong to $(0,1]\cap \mathbb{Q}$.

If there is a closed point $0\in U$ such that the fibre $(X_0,\Delta_0)$ has a good
minimal model then the restriction morphism
\[
\map {\pi_*\ring X.(m(K_X+\Delta))}.{H^0(X_u,\ring X_u.(m(K_{X_u}+\Delta_u)))}.
\]
is surjective for any $m\in\mathbb{N}$ such that $m\Delta$ is integral and for any closed
point $u\in U$.

In particular if $\psi\colon\rmap X.Z.$ is the ample model of $(X,\Delta)$ then
$\psi_u\colon\rmap X_u.Z_u.$ is the ample model of $(X_u,\Delta_u)$ for every closed point
$u\in U$.
\end{corollary}

The moduli space of stable curves is one of the most intensively studied varieties.  The
moduli space of stable varieties of general type is the higher dimensional analogue of the
moduli space of curves.  Unfortunately constructing this moduli space is more complicated
than constructing the moduli space of curves.  In particular it does not seem easy to use
GIT to construct the moduli space in higher dimensions; for example see \cite{WX14} for a
precise example of how badly behaved the situation can be.  Instead Koll\'ar and
Shepherd-Barron initiated a program to construct the moduli space in all dimensions in
\cite{KSB88}.  This program was carried out in large part by Alexeev for surfaces,
\cite{Alexeev94} and \cite{Alexeev96}.

We recall the definition of the moduli functor.  For simplicity, in the definition of the
functor, we restrict ourselves to the case with no boundary.  We refer to the forthcoming
book \cite{Kollar10} for a detailed discussion of this subject and to \cite{Kollar10a} for
a more concise survey.
\begin{definition}[Moduli of slc models, cf. {\cite[Definition 29]{Kollar10a}}]
Let $H(m)$ be an integer valued function. The moduli functor of semi log canonical models
with Hilbert function $H$ is
\[
\mslc_H(S)=\left\{
\begin{aligned}
&\text{flat projective morphisms $\map X.S.$, whose }\\ 
&\text{fibres are slc models with ample canonical class  	}\\
&\text{and Hilbert function $H(m)$, $\omega_X$ is flat over $S$ and}\\
&\text{all reflexive powers of $\omega_X$ commute with base change.}\\	
\end{aligned}
\right\}
\]
\end{definition}

In this paper we focus on the problem of showing that the moduli functor is bounded, so
that if we fix the degree, we get a bounded family.  The precise statement is given in
\eqref{t_modulifinite}.  We now describe the proof of \eqref{t_modulifinite}.  We first 
explain how to reduce to \eqref{t_one-smooth}.  

For curves if one fixes the genus $g$ then the moduli space is irreducible.  In particular
stable curves are always limits of smooth curves.  This fails in higher dimensions, so
that there are components of the moduli space whose general point corresponds to a
non-normal variety, or better, a semi log canonical variety.

Fortunately, cf. \cite[23, 24]{Kollar11} and \cite[5.13]{Kollar13}, one can reduce
boundedness of semi log canonical pairs to boundedness of log canonical pairs in a
straightforward manner.  If $(X,\Delta)$ is semi log canonical then let $n\colon\map Y.X.$
be the normalisation.  $X$ has nodal singularities in codimension one, so that informally
$X$ is obtained from $Y$ by identifying points of the double locus, the closure of the
codimension one singular locus.  More precisely, we may write
\[
K_Y+\Gamma=n^*(K_X+\Delta),
\]
where $\Gamma$ is the sum of the strict transform of $\Delta$ plus the double locus and
$(Y,\Gamma)$ is log canonical.  If $K_X+\Delta$ is ample then $(X,\Delta)$ is determined
by $(Y,\Gamma)$ and the data of the involution $\tau\colon\map S.S.$ of the normalisation
of the double locus.  Note that the involution $\tau$ fixes the different, the divisor 
$\Theta$ defined by adjunction in the following formula:
\[
(K_Y+\Gamma)|_S=K_S+\Theta.
\]

Conversely, if $(Y,\Gamma)$ is log canonical, $K_Y+\Gamma$ is ample, $\tau$ is an
involution of the normalisation $S$ of a divisor supported on $\rfdown\Gamma.$ which fixes
the different, then we may construct a semi log canonical pair $(X,\Delta)$, whose
normalisation is $(Y,\Gamma)$ and whose double locus is $S$.  

Note that $\tau$ fixes the pullback $L$ of the very ample line bundle determined by a
multiple of $K_X+\Delta$.  The group of all automorphisms of $S$ which fixes $L$ is a
linear algebraic group.  It follows, by standard arguments, that if $(Y,\Gamma)$ is
bounded then $\tau$ is bounded.  

Thus to prove \eqref{t_modulifinite} it suffices to prove the result when $X$ is normal,
that is, when $(X,\Delta)$ is log canonical, cf. \eqref{p_modulifinite}.  The first
problem is that a priori $X$ might have arbitrarily many components.  Note that if $X=C$
is a curve of genus $g$ then $K_X$ has degree $2g-2$ and so $X$ has at most $2g-2$
components.  In higher dimensions the situation is more complicated since $K_X$ is not
necessarily Cartier and so $d$ is not necessarily an integer.

Instead we use \cite[1.3.1]{HMX12}, which was conjectured by Alexeev \cite{Alexeev94} and
Koll\'ar \cite{Kollar92b}:
\begin{theorem}\label{t_dcc} Fix a positive integer $n$ and a set $I\subset [0,1]$
which satisfies the DCC.  Let $\mathfrak{D}$ be the set of log canonical pairs
$(X,\Delta)$ such that the dimension of $X$ is $n$ and the coefficients of $\Delta$ belong
to $I$.

Then the set
$$
\{\, \vol(X,K_X+\Delta) \,|\, (X,\Delta)\in \mathfrak{D} \,\},
$$
also satisfies the DCC.
\end{theorem}

Since there are only finitely many ways to write $d$ as a sum of elements $\llist d.k.$
taken from a set which satisfies the DCC, cf. \eqref{l_sum}, we are reduced to proving
\eqref{t_modulifinite} when $X$ is normal and irreducible.  

Let $\mathfrak{F}\subset \mathfrak{F}_{\text{slc}}(n,d,I)$ be the subset of all log
canonical pairs $(X,\Delta)$ where $X$ is irreducible.  Since the coefficients of $\Delta$
belong to a set which satisfies the DCC, \cite[1.3.3]{HMX12} implies that some fixed
multiple of $K_X+\Delta$ defines a birational map to projective space.  As the degree of
$K_X+\Delta$ is bounded by assumption, $\mathfrak{F}$ is log birationally bounded, that
is, there is a log pair $(Z,B)$ and a projective morphism $\pi\colon\map Z.U.$, such that
given any $(X,\Delta)\in\mathfrak{F}$, we may find $u\in U$ such that $X$ is birational to
$Z_u$ and the strict transform $\Phi$ of $\Delta$ plus the exceptional divisors are
components of $B_u$.

To fix ideas, it might help to introduce an example to illustrate some of the ideas that
go into the proof that $\mathfrak{F}$ is bounded.  We start with $\pr 2.$ and $k\geq 4$
lines.  The subscript $_0$ will indicate we are working with this example.  The variety
$U_0$ is the set of all configuations of $k$ lines, $Z_0=\pr 2.\times U_0$ and $B_0$ is
the reduced divisor corresponding to the lines.  We take $I_0=\{1/2,1\}$.

\cite[1.6]{HMX12} proves that $\mathfrak{F}$ is a bounded family provided if in addition
we assume that the total log discrepancy of $(X,\Delta)$ is bounded away from zero
(meaning that the coefficients of $\Delta$ are bounded away from one as well as the log
discrepancy is bounded away from zero).  For applications to moduli this is far too
strong; the double locus occurs with coefficient one.

Instead we proceed as follows.  By standard arguments we may assume that $U$ is smooth,
the morphism $\pi$ is smooth and its restriction to any strata of $B$ is smooth, that is,
$(Z,B)$ is log smooth over $U$.  In the case of lines in $\pr 2.$, we simply replace $U_0$
by the open subset of lines in linear general position; the case of lines not in general
position is handled by Noetherian induction.  We first reduce to the case when
$\vol(Z_u,K_{Z_u}+\Phi)=d$.  We are looking for a higher model $\map Y.Z.$ such that if
$C$ is the strict transform of $B$ plus the exceptionals and $u$ is a point then
$\vol(Y_u,K_{Y_u}+\Gamma)=d$ where $\Gamma$ is the transform of $\Delta$ plus the
exceptionals.  At this point we use some of the ideas that go into the proof of
\cite[1.9]{HMX10}.

We describe how to reduce to the case when $U$ is a point.  We illustrate the argument for
lines in $\pr 2.$; the argument in the general case is very similar.  In this case the
elements $(X,\Delta)\in \mathfrak{F}$ are constructed as follows.  Start with $\pr 2.$ and
a collection of $k$ lines in general position.  Let $\map S.{\pr 2.}.$ be any sequence of
smooth blow ups and let $D$ be the strict transform of the lines plus the exceptional
divisors.  Now blow down some $-1$-curves on $S$ to obtain $X$.  Let $\Delta$ be any
divisor supported on the pushforward of $D$ whose coefficients are $0$, $1/2$ or $1$.
Note that there are some restrictions on which $-1$-curves we blow down; we are only
allowed to blow down components of $D$ and we are also assuming that $(X,\Delta)$ is log
smooth.

To proceed further we want to understand how the volume changes for one smooth blow up of
a smooth surface $\pi\colon\map T.S.$.  Working locally, we may assume that $S=\af 2.$,
$D$ is the sum of the two coordinate lines $L_1$ and $L_2$ and $\pi$ blows up the origin.
Let $E$ be the exceptional divisor and let $M_1$ and $M_2$ be the strict transform of the
two lines.  By assumption $\Delta=a_1L_1+a_2L_2$ where $a_i=0$, $1/2$ or $1$.  If we write
\[
K_T+a_1M_1+a_2M_2+eE=\pi^*(K_S+a_1L_1+a_2L_2),
\]
then $e=a_1+a_2-1$.  

Globally we have a pair $(T,\Theta)$ such that $\pi_*\Theta=\Delta$.  If the volume of the
pair $(T,\Theta)$ is smaller than the volume of the pair $(S,\Delta)$ then the coefficient
$E$ of $\Theta$ is smaller than $e$.

In particular, since $e\leq 1$, if we increase the coefficient of any $-1$-curve we blow
down $\map S.X.$ to $1$ then the volume is unchanged.  So there is no harm in assuming
that $S=X$.  Note also that if we blow up $\map T.S.$ a point which does not belong to $D$
then $e\leq 0$ so that the volume is unchanged.  Therefore we may also assume that
$\map X.Z.$ only blows up strata of a fibre of $B$, since blow ups away from the strata
don't change the volume.  Since $(Z,B)$ is log smooth over $U$, any sequence of blows up
of the strata of a particular fibre can be realised in the whole family.  By deformation
invariance of log plurigenera we may therefore assume that $U$ is a point, \eqref{l_dcc}.

In general $\vol(Z_u,K_{Z_u}+\Phi)\geq \vol(X,K_X+\Delta)=d$.  Our goal is to find a
higher model $\map Y.Z.$ where we always have equality.  This follows using some results
from \cite{HMX10}, cf.  \eqref{l_point}.  We give an example at the end of \S 1 which
illustrates some of the subtleties behind the statement and proof of \eqref{l_point}.

So we may assume that $\vol(Z_u,K_{Z_u}+\Phi)=d$.  Since $(X,\Delta)$ is log canonical and
$K_X+\Delta$ is ample, we can recover $(X,\Delta)$ from $(Z_u,\Phi)$ as the log canonical
model, cf. \eqref{l_vol}.  Conversely if $u\in U$ is a point such that $(Z_u,\Phi)$ has a
log canonical model, $f\colon\rmap Z_u.X.$ , where
\[
X=\Proj R(Z_u,K_{Z_u}+\Phi) \qquad \text{and} \qquad \Delta=f_*\Phi,
\]
the coefficients of $0\leq \Phi\leq B_u$ belong to $I$ and $\vol(Z_u,K_{Z_u}+\Phi)=d$ then
$(X,\Delta)\in\mathfrak{F}$.

It therefore suffices to prove that the set of fibres with a log canonical model is
constructible.  Note that $(X,\Delta)$ has a log canonical model if and only if the log
canonical section ring
\[
R(X,K_X+\Delta)=\bigoplus_{m\in \mathbb{N}} H^0(X,\ring X.(m(K_X+\Delta)))
\]
is finitely generated.  Conjecturally every fibre has a log canonical model.  Once again
the problem are the components of $\Delta$ with coefficient one.  The main result of
\cite{BCHM10} implies that if there are no components of $\Delta$ with coefficient one,
that is, $(X,\Delta)$ is kawamata log terminal, then the log canonical section ring is
finitely generated.

In general, \eqref{l_equivalent}, the existence of the log canonical model $Z$ is
equivalent to the existence of a good minimal model $f\colon\rmap X.Y.$, that is, a model
$(Y,\Gamma)$ such that $K_Y+\Gamma$ is semi-ample.  In this case the log canonical model
is simply the model $\map Y.Z.$ such that $K_Y+\Gamma$ is the pullback of an ample
divisor.

In fact we prove, \eqref{t_one-smooth}, a much stronger result.  We prove that if one
fibre $(X_0,\Delta_0)$ has a good minimal model then every fibre has a good minimal model.
By \cite[1.1]{HX11} it suffices to prove that every fibre over an open subset has a good
minimal model, equivalently, that the generic fibre has a good minimal model.

Let $\eta\in U$ be the generic point.  We may assume that $U$ is affine.  We prove the
existence of a good minimal model for the pair $(X_{\eta},\Delta_{\eta})$ in two steps.
We first show that $(X_{\eta},\Delta_{\eta})$ has a minimal model.  For this we run the
$(K_X+\Delta)$-MMP with scaling of an ample divisor.  We know that if we run the
$(K_{X_0}+\Delta_0)$-MMP with scaling of an ample divisor then this MMP terminates with a
good minimal model.  Using \cite[2.10]{HX11} and \eqref{l_base} we can reduce to the case
when the diminished stable base locus of $K_{X_0}+\Delta_0$ does not contain any non
canonical centres.  In this case we show, \eqref{l_central}, that every step of the
$(K_X+\Delta)$-MMP induces a $(K_{X_0}+\Delta_0)$-negative map.  This generalises
\cite[4.1]{HMX10}, which assumes that $U$ is a curve and that $(X,\Delta)$ is terminal.
This MMP ends $f\colon\rmap X.Y.$ with a minimal model for the generic fibre,
\eqref{l_generic-minimal}.

To finish off we need to show that the minimal model is a good minimal model.  There are
two cases.  We may write $(X,\Delta=S+B)$, where $S=\rfdown\Delta.$.  

In the first case, if $K_X+(1-\epsilon)S+B$ is not pseudo-effective for any $\epsilon>0$
then we may run $\rmap Y.W.$ the $(K_X+(1-\epsilon)S+B)$-MMP until we reach a Mori fibre
space \eqref{l_generic-fano} $\map W.Z.$.  If $\epsilon>0$ is sufficiently small, this MMP
induces a $(K_{X_0}+\Delta_0)$-non-positive map, see \eqref{l_index}.  It follows that
this MMP is $(K_X+\Delta)$-non-positive.  We know that there is a component $D$ of $S$
whose image dominates the base $Z$ of the Mori fibre space.  By induction the generic
fibre of the image $E$ of $D$ in $Y$ is a good minimal model.  The restriction $\rmap
E.F.$ of the map $\rmap Y.W.$ need not be a birational contraction but we won't lose
semi-ampleness.  The image of the divisor is pulled back from $Z$ and so $K_X+\Delta$ has
a semi-ample model.

In the second case $K_X+(1-\epsilon)S+B$ is pseudo-effective.  As $K_X+(1-\epsilon)S+B$ is
kawamata log terminal, it follows by work of B. Berndtsson and M. P\u aun, \eqref{t_BP},
that the Kodaira dimension is invariant, see \eqref{t_inv}.  As $K_X+(1-\epsilon)S+B$ is
pseudo-effective and $(X_0,\Delta_0)$ has a good minimal model, it follows that
$K_{X_0}+\Delta_0$ is abundant, that is, the Kodaira dimension is the same as the
numerical dimension.  By deformation invariance of log plurigenera the generic fibre is
abundant.  As the restriction of $K_Y+\Gamma$ to every component of coefficient one is
semi-ample, the restriction of $K_Y+\Gamma$ to the sum of the coefficient one part is
semi-ample by \eqref{t_slc} and we are done by \eqref{t_bpf}.

As promised, here is an example to illustrate some of the subtleties of the argument in
the proof of \eqref{l_point}.  We go back to the example of lines in $\pr 2.$.  We start
with four lines $L_1$, $L_2$, $L_3$ and $L_4$ in $\pr 2.$ all with coefficient one.  In
this case $U_0$ is a point since there is no moduli to four lines in linear general
position.  The volume of the pair in $\pr 2.$ is then $1$.  Now suppose that
$(X,\Delta)\in \mathfrak{F}$.  As already pointed out $d\leq 1$ and there is no harm in
assuming that $X$ is a blow up of $\pr 2.$, $f\colon\map X.{\pr 2.}.$.  We may even assume
that all of the blow ups lie over the six points where the four lines intersect.  Fix the
point $p$ where the two lines $L_1$ and $L_2$ meet and assume that all blow ups are over
$p$.  Then $X$ is a toric variety and the morphism $f\colon\map X.{\pr 2.}.$ is a toric
morphism.  Let's simplify matters even more and assume that we only alter one coefficient
of one exceptional divisor $E$ over $p$; let's suppose that we don't include $E$ in
$\Theta$, that is, we make its coefficient zero.  In this case, since every other divisor
occurs with coefficient one, we can compute the volume on the weighted blow up of $\pr 2.$
corresponding to the divisor $E$, $g\colon\map S.{\pr 2.}.$.  The problem is that unless
we fix the degree $d$ there is no constraint on how many times we blow up over $p$, that
is, there is no constraint on the weighted blow up $g$.  Let $M_1$, $M_2$, $M_3$ and $M_4$
be the strict transform of the four lines.  Then $(S,M_1+M_2+M_3+E)$ is a toric pair, so
that $K_S+M_1+M_2+M_3+E \sim 0$.  It follows that
\[
(K_S+M_1+M_2+M_3+M_4)^2=(M_4-E)^2=M_4^2+E^2=1+E^2.
\]
It is a simple exercise in toric geometry to compute $E^2$.  If we make a weighted blow up of 
type $(a,b)$ then 
\[
E^2=-\frac 1{ab}
\]
so that the volume is
\[
\frac{ab-1}{ab}.   
\]
As expected the volume satisfies DCC.  If we fix the volume $d$ then there are only
finitely many possible values for $(a,b)$.  This is the content of \eqref{l_point} in this
example.

\makeatletter
\renewcommand{\thetheorem}{\thesubsection.\arabic{theorem}}
\@addtoreset{theorem}{subsection}
\makeatother

\section{Preliminaries}

\subsection{Notations and Conventions}

We will follow the terminology from \cite{KM98}.  Let $f\colon\rmap X.Y.$ be a proper
birational map of normal quasi-projective varieties and let $p\colon\map W.X.$ and
$q\colon\map W.Y.$ be a common resolution of $f$.  We say that $f$ is a \textit{birational
  contraction} if every $p$-exceptional divisor is $q$-exceptional.  If $D$ is an
$\mathbb{R}$-Cartier divisor on $Y$ then $f^*D$ is the $\mathbb{R}$-Weil divisor
$q_*p^*D$.  Equivalently, if $U$ is the domain of $f$ then $f^*D$ is the $\mathbb{R}$-Weil
divisor on $X$ corresponding to the $\mathbb{R}$-Cartier divisor $(f|_U)^*D$ on $U$.

If $D$ is an $\mathbb{R}$-Cartier divisor on $X$ such that $D':=f_*D$ is
$\mathbb{R}$-Cartier then we say that $f$ is $D$-\textit{non-positive}
(resp. $D$-\textit{negative}) if we have $p^*D=q^*D'+E$ where $E\geq 0$ and $E$ is
$q$-exceptional (respectively $E$ is $q$-exceptional and the support of $E$ contains the
strict transform of the $f$-exceptional divisors).

We say a proper morphism $\pi\colon\map X.U.$ is a \textit{contraction morphism} if
$\pi_*\ring X.=\ring U.$.  Recall that for any $\mathbb{R}$-divisor $D$ on $X$, the sheaf
$\pi_*\ring X.(D)$ is defined to be $\pi_*\ring X.(\rfdown D.)$.  

If $X$ is a normal variety and $B$ is a divisor whose components have coefficient one then
the \textit{strata} of $B$ are the irreducible components of the intersections
\[
B_I=\cap _{j\in I} B_j=B_{i_1}\cap \ldots \cap B_{i_r},
\]
of the components of $B$, where $I=\{\, \llist i.r.\,\}$ is a subset of the indices,
including the empty intersection $X=B_\emptyset$.  If $(X,\Delta)$ is a log pair then the
\textit{strata} of $(X,\Delta)$ are the strata of the support $B$ of $\Delta$.

If we are given a morphism $\map X.U.$, then we say that $(X,\Delta)$ is \textit{log
  smooth over $U$} if $(X,\Delta)$ has simple normal crossings and both $X$ and every
stratum of $(X,D)$ is smooth over $U$, where $D$ is the support of $\Delta$.  If
$\pi\colon\map X.U.$ and $\map Y.U.$ are projective morphisms, $f\colon\rmap X.Y.$ is a
birational contraction over $U$ and $(X,\Delta)$ is a log canonical pair (respectively
divisorially log terminal $\mathbb{Q}$-factorial pair) such that $f$ is
$(K_X+\Delta)$-non-positive (respectively $(K_X+\Delta)$-negative) and $K_Y+\Gamma$ is nef
over $U$ (respectively and $Y$ is $\mathbb{Q}$-factorial), then we say that
$f\colon\rmap X.Y.$ is a \textit{weak log canonical model} (respectively \textit{a minimal
  model}) of $K_X+\Delta$ over $U$.

We say $K_Y+\Gamma$ is \textit{semi-ample} over $U$ if there exists a contraction morphism
$\psi\colon\map Y.Z.$ over $U$ such that $K_Y+\Gamma\sim_{\mathbb{R}}\psi ^*A$ for some
$\mathbb{R}$-divisor $A$ on $Z$ which is ample over $U$.  Equivalently, when $K_Y+\Gamma$
is $\mathbb{Q}$-Cartier, $K_Y+\Gamma$ is semi-ample over $U$ if there exists an integer
$m>0$ such that $\ring Y.(m(K_Y+\Gamma))$ is generated over $U$.  Note that in this case
\[
R(Y/U,K_Y+\Gamma):=\bigoplus_{m\geq 0}\pi_*\ring Y.(m(K_Y+\Gamma))
\]
is a finitely generated $\ring U.$-algebra, and 
\[
Z=\Proj R(Y/U,K_Y+\Gamma).
\]  
If $K_Y+\Gamma$ is semi-ample and big over $U$, then $Z$ is the \textit{log canonical
  model} of $(X,\Delta)$ over $U$.  A weak log canonical model $f\colon\rmap X.Y.$ is
called a \textit{semi-ample model} if $K_Y+\Gamma$ is semi-ample.

Suppose that $\pi\colon\map X.U.$ is a projective morphism of normal varieties.  Let $D$
be an $\mathbb{R}$-Cartier divisor on $X$.  Let $C$ be a prime divisor.  If $D$ is big
over $U$ then
\[
\sigma_C(X/U,D)=\inf \{\, \mult_C(D') \,|\, D'\sim_{\mathbb{R},U} D, D'\geq 0\,\}.
\]
Now let $A$ be any ample $\mathbb{Q}$-divisor over $U$ and suppose that $D$ is
pseudo-effective over $U$.  Following \cite{Nakayama04}, let
\[
\sigma_C(X/U,D)=\lim_{\epsilon\to 0} \sigma_C(X/U,D+\epsilon A).
\]
Then $\sigma_C(X/U,D)$ exists (where we allow $\infty$ as a limit) and is independent of
the choice of $A$.  There are only finitely many prime divisors $C$ such that
$\sigma_C(X/U,D)>0$, this number only depends on the numerical equivalence class of $D$
over $U$ and if we replace $U$ by an open subset which contains the image of the generic
point of $C$ then $\sigma_C$ is unchanged.  However with no more assumptions there are
examples when $\sigma_C(X/U,D)=\infty$, \cite{Lesieutre15}.  On the other hand if $\pi(C)$
has codimension no more than one then $\sigma_C(X/U,D)<\infty$.  In this case the
$\mathbb{R}$-divisor $N_{\sigma}(X/U,D)=\sum _C\sigma_C(D)C$ is determined by the
numerical equivalence class of $D$, cf. \cite[3.3.1]{BCHM10} and \cite{Nakayama04} for
more details.  Note that if the fibres of $\pi$ are irreducible and all of the same
dimension then $\pi(C)$ automatically has codimension at most one for every prime divisor
$C$ on $X$.

Now suppose that $D$ is only an $\mathbb{R}$-divisor.  The \textit{real linear system}
associated to $D$ over $U$ is
\[
|D/U|_{\mathbb{R}}=\{\, C\geq 0\,|\, C\sim_{\mathbb{R},U} D \,\}.
\]
The \textit{stable base locus} of $D$ over $U$ is the Zariski closed set
$\mathbf{B}(X/U,D)$ given by the intersection of the support of the elements of the real
linear system $|D/U|_{\mathbb{R}}$.  If $|D/U|_{\mathbb{R}}=\emptyset$, then we let
$\mathbf{B}(X/U,D)=X$.

The \textit{diminished stable base locus of $D$ over $U$} is
\[
\mathbf{B}_-(X/U,D)=\bigcup_{A} \mathbf{B}(X/U,D+A),
\]
where the union runs over all divisors $A$ which are ample over $U$.

Suppose that $U$ is a point.  Following \cite{Nakayama04} if $D$ is pseudo-effective we
define \textit{the numerical dimension}
\[
\kappa_{\sigma}(X,D)=\max_{H\in {\Pic}(X)}\{\, k\in \mathbb{N} \,|\, \limsup_{m\to \infty} \frac{h^0(X,\ring X.(mD+H))}{m^k} >0\,\}.
\]
If $D$ is nef then this is the same as 
\[
\nu(X,D)=\max \{\, k\in \mathbb{N} \,|\, H^{n-k}\cdot D^k>0 \,\}
\]
for any ample divisor $H$, see \cite{Nakayama04}.  If $D$ is $\mathbb{Q}$-Cartier then $D$
is called \textit{abundant} if $\kappa_{\sigma}(X,D)=\kappa(X,D)$, that is, the numerical
dimension is equal to the Iitaka dimension.  If we drop the condition that $X$ is
projective and instead we have a projective morphism $\pi\colon\map X.U.$, then a
$\mathbb{Q}$-Cartier divisor $D$ on $X$, is called \textit{abundant} over $U$ if its
restriction to the generic fibre is abundant.

If $(X,\Delta)$ is a log pair then a \textit{non canonical centre} is the centre of a
valuation of log discrepancy less than one.

We say a family of log pairs $\mathfrak{D}$ is \textit{bounded} if there is a morphism of
varieties $\map Z.U.$, where $U$ is smooth, $Z$ is flat over $U$, and a log pair
$(Z,\Sigma)$, where the support of $\Sigma$ contains neither a component of a fibre nor a
codimension one singular point of any fibre, such that for every
$(X,\Delta)\in \mathfrak{D}$ there is a closed point $u\in U$ and an isomorphism of log
pairs between $(X,\Delta)$ and $(Z_u,\Sigma_u)$.  In particular the coefficients of
$\Delta$ belong to a finite set.

\subsection{The volume}

\begin{definition}\label{d_volume} Let $X$ be a normal $n$-dimensional irreducible
projective variety and let $D$ be an $\mathbb{R}$-divisor.  The \textbf{volume} of $D$ is
\[
\vol(X,D) = \limsup_{m\to \infty}\frac{n! h^0(X,\ring X.(mD))}{m^n}.
\]
Let $V\subset X$ be a normal irreducible subvariety of dimension $d$.  Suppose that $D$ is
$\mathbb{R}$-Cartier whose support does not contain $V$.  The \textbf{restricted volume}
of $D$ along $V$ is
\[
\vol(X|V,D)=\limsup_{m\to \infty}\frac{d! (\dim \im (\map {H^0(X,\ring X.(mD))}.{H^0(V,\ring V.(mD|_V))}.))}{m^d}.
\]
\end{definition}

\begin{lemma}\label{l_vol} Let $f\colon\map X.Z.$ be a birational morphism between 
log canonical pairs $(X,\Delta)$ and $(Z,B)$.  Suppose that $K_X+\Delta$ is big and that
$(X,\Delta)$ has a log canonical model $g\colon\rmap X.Y.$.

If $f_*\Delta\leq B$ and $\vol(X,K_X+\Delta)=\vol(Z,K_Z+B)$ then the induced birational 
map $\rmap Z.Y.$ is the log canonical model of $(Z,B)$.  
\end{lemma}
\begin{proof} Let $\pi\colon\map W.X.$ be a log resolution of $(X,C+F)$, which also
resolves the map $g$, where $C$ is the strict transform of $B$ and $F$ is the sum of the
$f$-exceptional divisors.  We may write
\[
K_W+\Theta=\pi^*(K_X+\Delta)+E,
\]
where $\Theta\geq 0$ and $E\geq 0$ have no common components, $\pi_*\Theta=\Delta$ and
$\pi_*E=0$.  Then the log canonical model of $(W,\Theta)$ is the same as the log canonical
model of $(X,\Delta)$.  Replacing $(X,\Delta)$ by $(W,\Theta)$ we may assume that
$(X,C+F)$ is log smooth and $g\colon\map X.Y.$ is a morphism.  Replacing $(Z,B)$ by the
pair $(X,D=C+F)$, we may assume $Z=X$.  

If $A=g_*(K_X+\Delta)$ and $H=g^*A$ then $A$ is ample and $K_X+\Delta-H\geq 0$.  Let
$L=D-\Delta\geq 0$, let $S$ be a component of $L$ with coefficient $a$ and let
\[
v(t)=\vol(X,H+tS).
\]
Then $v(t)$ is a non-decreasing function of $t$ and 
\begin{align*} 
v(0)&=\vol(X,H)\\
    &=\vol(X,K_X+\Delta)\\
    &=\vol(X,K_X+D)\\
    &\geq \vol(X,H+L)\\
    &\geq \vol(X,H+aS)\\
    &=v(a).
\end{align*} 
Thus $v(t)$ is constant over the range $[0,a]$.  \cite[4.25 (iii)]{LM09} implies that
\[
\frac 1n \deriv v.t.\bigg|_{t=0}=\vol_{X|S}(H)\geq S\cdot H^{n-1}=g_*S\cdot A^{n-1}
\]
so that $g_*S=0$.  But then every component of $L$ is exceptional for $g$ and $g$ is the
log canonical model of $(X,D)$.
\end{proof}

\subsection{Deformation Invariance}

\begin{lemma}\label{l_lift} Let $\pi\colon\map X.U.$ be a projective morphism to a smooth
variety $U$ and let $(X,\Delta)$ be a log smooth pair over $U$.  Let $A$ be a relatively
ample Cartier divisor such that $\rfdown \Delta.+A \sim A'$ where $(X,\Delta+A')$ is log
smooth over $U$.

If the coefficients of $\Delta$ belong to $[0,1]$ then 
\[
\map {f_*\ring X.(m(K_X+\Delta)+A)}.{H^0(X_u,\ring X_u.(m(K_{X_u}+\Delta_u)+A_u))}.
\] 
is surjective for all positive integers $m$ such that $m\Delta$ is integral and for every
$u\in U$.
\end{lemma}
\begin{proof} We have 
\[
m(K_X+\Delta)+A \sim m(K_X+\Delta-\frac 1m\rfdown \Delta.+\frac 1m A')=m(K_X+\Delta'),
\]
where $(X,\Delta')$ is log smooth over $U$, $\rfdown\Delta'.=0$ and $\Delta'$ is big over
$U$ so that we may apply \cite[1.8.1]{HMX10}.
\end{proof}

\begin{lemma}\label{l_def} Let $\pi\colon\map X.U.$ be a projective morphism to a smooth
variety $U$ and let $(X,\Delta)$ be a log smooth pair over $U$.  Assume that $K_X+\Delta$
is pseudo-effective over $U$.

If the coefficients of $\Delta$ belong to $[0,1]$ then 
\[
N_{\sigma}(X/U,K_X+\Delta)|_{X_u}=N_{\sigma}(X_u,K_{X_u}+\Delta_u)
\]
for every $u\in U$.  
\end{lemma}
\begin{proof} Since this result is local about every point $u\in U$ we may assume that $U$
is affine.  Pick a relatively ample Cartier divisor $A$ such that
$\rfdown \Delta.+A \sim A'$ where $(X,\Delta+A')$ is log smooth over $U$.  Fix $u\in U$.  Then
\eqref{l_lift} implies that
\[
\map {f_*\ring X.(m(K_X+\Delta)+A)}.{H^0(X_u,\ring X_u.(m(K_{X_u}+\Delta_u)+A_u))}.
\] 
is surjective for all positive integers $m$ such that $m\Delta$ is integral.  It follows
that
\[
N_{\sigma}(X/U,K_X+\Delta)|_{X_u}\leq N_{\sigma}(X_u,K_{X_u}+\Delta_u)
\]
and the reverse inequality is clear.
\end{proof}

\begin{lemma}\label{l_commute} Let $\pi\colon\map X.U.$ be a projective morphism to a smooth
variety $U$ and let $(X,\Delta)$ be a log smooth pair over $U$ such that the strata of
$\Delta$ have irreducible fibers over $U$ and $K_X+\Delta$ is pseudo-effective over $U$.
Let $0\in U$ be a closed point, let
\[
\Theta_0=\Delta_0-\Delta_0\wedge N_{\sigma}(X_0,K_{X_0}+\Delta_0)
\]
and let $0\leq \Theta\leq\Delta$ be the unique divisor so that $\Theta_0=\Theta|_{X_0}$.  

If the coefficients of $\Delta$ belong to $[0,1]$ then 
\[
\Theta=\Delta-\Delta\wedge N_{\sigma}(X/U,K_X+\Delta).
\]
\end{lemma}
\begin{proof} Replacing $U$ be an open neighbourhood of $0\in U$ we may assume that $U$ is
affine.  Pick a relatively ample Cartier divisor $H$ with the property that for every
integral divisor $0\leq S\leq \rfdown\Delta.$ we may find $S+H \sim H'$ such that
$(X,\Delta+H')$ is log smooth over $U$.  Given a positive integer $m$, let
\[
\Phi_0=\Delta_0-\Delta_0\wedge N_{\sigma}(X_0,K_{X_0}+\Delta_0+\frac 1m H_0)
\]
and let $0\leq \Phi\leq\Delta$ be the unique divisor so that $\Phi_0=\Phi|_{X_0}$.
Consider the commutative diagram
\[
\begin{diagram}
\pi_*\ring X.(m(K_X+\Phi)+H)   &   \rTo    & \pi_*\ring X.(m(K_X+\Delta)+H)  \\
\dTo &  & \dTo \\
H^0(X_0,\ring X_0.(m(K_{X_0}+\Phi_0)+H_0))   &   \rTo    & H^0(X_0,\ring {X_0}.(m(K_{X_0}+\Delta_0)+H_0). 
\end{diagram}
\] 
The top row is an inclusion and the bottom row is an isomorphism by assumption.  The first
column is surjective by \eqref{l_lift}.  Nakayama's Lemma implies that the top row is an
isomorphism in a neighbourhood of $X_0$.  It follows that
\[
\Phi\geq \Delta-\Delta\wedge N_{\sigma}(X/U,K_X+\Delta+\frac 1mH).
\]
Taking the limit as $m$ goes to infinity we get 
\[
\Theta\geq \Delta-\Delta\wedge N_{\sigma}(X/U,K_X+\Delta)
\]
and the reverse inequality follows by \eqref{l_def}.  
\end{proof}

\begin{lemma}\label{l_one} Let $\pi\colon\map X.U.$ be a projective morphism to a smooth 
variety $U$ and let $(X,D)$ be log smooth over $U$, where the coefficients of $D$ are all
one.  Let $0\in U$ be a closed point.

Then the restriction morphism
\[
\map {\pi_*\ring X.(K_X+D)}.{H^0(X_0,\ring X_0.(K_{X_0}+D_0))}.
\]
is surjective.  
\end{lemma}
\begin{proof} Since the result is local we may assume that $U$ is affine.  Cutting by
hyperplanes we may assume that $U$ is a curve.  Thus we want to show that the restriction
map
\[
\map {H^0(X,\ring X.(K_X+X_0+D))}.{H^0(X_0, \ring X_0.(K_{X_0}+D_0))}.
\]
is surjective.  This is equivalent to showing that multiplication by a local parameter
\[
\map {H^1(X,\ring X.(K_X+D))}.{H^1(X,\ring X.(K_X+D+X_0))}.
\]
is injective.

By assumption the image of every strata of $D$ is the whole of $U$ and $0=(K_X+D)-(K_X+D)$
is semi-ample.  Therefore a generalisation of Koll\'ar's injectivity theorem (see
\cite{Kollar86}, \cite[6.3]{Fujino09} and \cite[5.4]{Ambro14}) implies that
\[
\map {H^1(X,\ring X.(K_X+D))}.{H^1(X,\ring X.(K_X+D+X_0))}.
\]
is injective.
\end{proof}

\subsection{DCC sets}

\begin{lemma}\label{l_sum} Let $I\subset \mathbb{R}$ be a set of positive real 
numbers which satisfies the DCC.  Fix a constant $d$.  

Then the set 
\[
T=\{\, (\llist d.k.)\,|\, k\in \mathbb{N}, d_i\in I, \sum d_i=d \,\}
\]
is finite.
\end{lemma}
\begin{proof} As $I$ satisfies the DCC there is a real number $\delta>0$ such that if
$i\in I$ then $i\geq \delta$.  Thus 
\[
k\leq \frac d{\delta}.
\]
It is enough to show that given any infinite sequence $\ilist t.$ of elements of $T$ that
we may find a constant subsequence.  Possibly passing to a subsequence we may assume that
the number of entries $k$ of each vector $t_i=(\llist {d_i}.k.)$ is constant.  Since $I$
satisfies the DCC, possibly passing to a subsequence, we may assume that the entries
are not decreasing.  Since the sum is constant, it is clear that the entries are constant, 
so that $\ilist t.$ is a constant sequence.  
\end{proof}

\begin{lemma}\label{l_finite} Let $J$ be a finite set of real numbers at most one. 

If
\[
I=\{\, a\in (0,1] \,|\, a=1+\sum_{i\leq k} a_i-k, \llist a.k.\in J\,\}.
\]
then $I$ is finite.
\end{lemma}
\begin{proof} If $a_k=1$ then 
\[
\sum_{i\leq k} a_i-k=\sum_{i\leq k-1} a_i-(k-1).
\]
Thus there is no harm in assuming that $1\notin J$.   If $a_k<0$ then 
\[
1+\sum_{i\leq k} a_i-k<0.
\]
Thus we may assume that $J\subset [0,1)$.  

Note that 
\[
1+\sum_{i\leq k} a_i-k>0 \qquad \text{if and only if} \qquad \sum_{i\leq k} (1-a_i)<1.
\]
Since $J$ is finite we may find $\delta>0$ such that if $a\in J$ then $1-a\geq \delta$.  
This bounds $k$ and the result is clear. 
\end{proof}

\subsection{Semi log canonical varieties}

We will need the definition of certain singularities of semi-normal pairs,
\cite[7.2.1]{Kollar96}.  Let $X$ be a semi-normal variety which satisfies Serre's
condition $S_2$.  We say that $X$ is \textit{demi-normal} if $X$ has nodal singularities
in codimension one \cite[5.1]{Kollar13}.  Let $\Delta$ be an $\mathbb{R}$-divisor on $X$,
such that no component of $\Delta$ is contained in the singular locus of $X$ and such that
$K_X+\Delta$ is $\mathbb{R}$-Cartier.  Let $n\colon \map Y.X.$ be the normalisation of $X$
and write
\[
K_Y+\Gamma=n^*(K_X+\Delta),
\] 
where $\Gamma$ is the sum of the strict transform of $\Delta$ and the double locus.  We
say that $(X,\Delta)$ is \textit{semi log canonical} if $(Y,\Gamma)$ is log canonical.
See \cite{Kollar13} for more details about semi log canonical singularities.

\begin{theorem}\label{t_slc} Let $(X,\Delta)$ be a semi log canonical pair and let
$n\colon\map Y.X.$ be the normalisation.  By adjunction we may write
\[
K_Y+\Gamma=n^*(K_X+\Delta),
\]
where $(Y,\Gamma)$ is log canonical.  

If $X$ is projective and $\Delta$ is a $\mathbb{Q}$-divisor then $K_X+\Delta$ is
semi-ample if and only if $K_Y+\Gamma$ is semi-ample.
\end{theorem}
\begin{proof} See \cite{FG11} or \cite[1.4]{HX11a}.
\end{proof}

Suppose that $(X,\Delta)$ is log canonical and $\pi\colon\map X.U.$ is a morphism of
quasi-projective varieties.  Suppose that $U$ is smooth, the fibres of $\pi$ all have the
same dimension and the support of $\Delta$ does not contain any fibre.

If $(X_0,\Delta_0)$ is the fibre over a closed point $0\in U$ and $X_0$ is integral and
normal then note that
\[
(K_X+\Delta)|_{X_0}=K_{X_0}+\Delta_0.
\]

\subsection{Base Point Free Theorem}

Recall the following generalizsation of Kawamata's theorem:
\begin{theorem}\label{t_bpf} Let $(X,\Delta=S+B)$ be a divisorially log terminal pair, 
where $S=\rfdown \Delta.$ and $B$ is a $\mathbb{Q}$-divisor.  Let $H$ be a
$\mathbb{Q}$-Cartier divisor on $X$ and let $\map X.U.$ be a proper surjective morphism of
varieties.

If there is a constant $a_0$ such that 
\begin{enumerate}
\item $H|_S$ is semi-ample over $U$, 
\item $aH-(K_X+\Delta)$ is nef and abundant over $U$, for all $a>a_0$,
\end{enumerate}
then $H$ is semi-ample over $U$.
\end{theorem}
\begin{proof} See \cite{Kawamata85}, \cite[3.2]{Fukuda02}, \cite{Ambro05},
\cite{Fujino05}, \cite{Fujino09}, \cite{Fujino12}, \cite[4.1]{HX11} and \cite{FG11}.
\end{proof}

\subsection{Minimal models}

\begin{lemma}\label{l_weak} Let $(X,\Delta)$ be a log canonical pair, where $X$ is a 
projective variety and let $f\colon\rmap X.Y.$ be a weak log canonical model.  Suppose
that the rational map $\phi$ associated to the linear system $|r(K_X+\Delta)|$ is
birational.

Then,
\begin{enumerate} 
\item Every component of $N_{\sigma}(X,K_X+\Delta)$ is $f$-exceptional.
\item If $P$ is a prime divisor such that $P$ is not a component of the base locus of
$|r(K_X+\Delta)|$ and the restriction of $\phi$ to $P$ is birational then $P$ is not
$f$-exceptional.
\end{enumerate} 
\end{lemma}
\begin{proof} Let $p\colon\map W.X.$ and $q\colon\map W.Y.$ resolve $f$.  As $f$ is a weak
log canonical of $(X,\Delta)$, we may write
\[
p^*(K_X+\Delta)=q^*(K_Y+\Gamma)+E,
\]
where $E\geq 0$ is $q$-exceptional.  As $q^*(K_Y+\Gamma)$ is nef, it follows that
\[
N_{\sigma}(X,K_X+\Delta)=p_*E.
\]
In particular (1) holds.

If $Q$ is the strict transform of $P$ and $\psi$ is the birational map associated to the
linear system $|r p^*(K_X+\Delta)|$ then the restriction of the map $\psi$ to $Q$ is
birational.  On the other hand,
\[
|r p^*(K_X+\Delta)|=|r q^*(K_Y+\Gamma)|+rE.
\]
Therefore $\psi$ is the birational map associated to the linear system
$|r q^*(K_Y+\Gamma)|$.  In particular $Q$ is not $q$-exceptional so that $P$ is not
$f$-exceptional.
\end{proof}

\begin{lemma}\label{l_rational} Let $(X,\Delta)$ be a divisorially log terminal pair
where $X$ is $\mathbb{Q}$-factorial and projective.  Assume that $K_X+\Delta$ is
pseudo-effective.  Suppose that we run $f\colon\rmap X.Y.$ the $(K_X+\Delta)$-MMP with
scaling of an ample divisor $A$, so that $(Y,\Gamma+tB)$ is nef, where $\Gamma=f_*\Delta$
and $B=f_*A$.

\begin{enumerate} 
\item If $F$ is $f$-exceptional then $F$ is a component of $N_{\sigma}(X,K_X+\Delta)$.  
\item If $t>0$ is sufficiently small then every component of $N_{\sigma}(X,K_X+\Delta)$
is $f$-exceptional.  
\item If $(X,\Delta)$ has a minimal model and $K_X+\Delta$ is $\mathbb{Q}$-Cartier then
$N_{\sigma}(X,K_X+\Delta)$ is a $\mathbb{Q}$-divisor.
\end{enumerate} 
\end{lemma}
\begin{proof} Let $p\colon\map W.X.$ and $q\colon\map W.Y.$ resolve $f$.  As $f$ is a
minimal model of $(X,tA+\Delta)$, for some some $t\geq 0$, we may write
\[
p^*(K_X+tA+\Delta)=q^*(K_Y+tB+\Gamma)+E,
\]
where $E=E_t\geq 0$ is $q$-exceptional.  As $q^*(K_Y+tB+\Gamma)$ is nef, it follows that
\[
N_{\sigma}(X,K_X+tA+\Delta)=p_*E.
\]
As $A$ is ample, (1) holds.  If $t$ is sufficiently small then 
\[
N_{\sigma}(X,K_X+tA+\Delta) \qquad \text{and} \qquad N_{\sigma}(X,K_X+\Delta)
\]
have the same support and so (2) holds.  

If $(X,\Delta)$ has a minimal model then we may assume that $t=0$ and so
\[
N_{\sigma}(X,K_X+\Delta)=p_*E_0
\]
is a $\mathbb{Q}$-divisor.  
\end{proof}

\begin{lemma}\label{l_only} Let $(X,\Delta)$ be a divisorially log terminal pair
where $X$ is $\mathbb{Q}$-factorial and projective.  Assume that $K_X+\Delta$ is
pseudo-effective.  

If $f\colon\rmap X.Y.$ is a birational contraction such that $Y$ is
$\mathbb{Q}$-factorial, $K_Y+\Gamma=f_*(K_X+\Delta)$ is nef and $f$ only contracts
components of $N_{\sigma}(X,K_X+\Delta)$ then $f$ is a minimal model of $(X,\Delta)$.
\end{lemma}
\begin{proof} Let $p\colon\map W.X.$ and $q\colon\map W.Y.$ resolve $f$.  We may write
\[
p^*(K_X+\Delta)+E=q^*(K_Y+\Gamma)+F,
\]
where $E\geq 0$ and $F\geq 0$ have no common components and both $E$ and $F$ are
$q$-exceptional.  

As $K_Y+\Gamma$ is nef, the support of $F$ and the support of
$N_{\sigma}(W,q^*(K_Y+\Gamma)+F)$ coincide.  On the other hand, every component of $E$ is
a component of $N_{\sigma}(W,p^*(K_X+\Delta)+E)$.  Thus $E=0$ and any divisor contracted
by $f$ is a component of $F$.  \end{proof}

\subsection{Blowing up log pairs}

\begin{lemma}\label{l_klt} Let $(X,\Delta)$ be a log smooth pair.

If $\rfdown \Delta.=0$ then there is a sequence $\pi\colon\map Y.X.$ of
smooth blow ups of the strata of $(X,\Delta)$ such that if we write 
\[
K_Y+\Gamma=\pi^*(K_X+\Delta)+E,
\]
where $\Gamma\geq 0$ and $E\geq 0$ have no common components, $\pi_*\Gamma=\Delta$ and
$\pi_*E=0$, then no two components of $\Gamma$ intersect.
\end{lemma}
\begin{proof} This is standard, see for example \cite[6.5]{HM07a}.  
\end{proof}

\begin{lemma}\label{l_coefficients} Let $(X,\Delta)$ be a sub log canonical pair.  

We may find a finite set $I\subset (0,1]$ such that if $\pi\colon\map Y.X.$ is any
birational morphism and we write
\[
K_Y+\Gamma=\pi^*(K_X+\Delta) 
\]
then the coefficients of $\Gamma$ which are positive belong to $I$.
\end{lemma}
\begin{proof} Replacing $(X,\Delta)$ by a log resolution we may assume that $(X,\Delta)$
is log smooth.  Let $J$ be the set of coefficients of $\Delta$ and let $I$ be the set
given by \eqref{l_finite}.

Supppose that $\pi\colon\map Y.X.$ is a birational morphism.  We may write
\[
K_Y+\Gamma=\pi^*(K_X+\Delta).
\]
We claim that the coefficients of $\Gamma$ which are positive belong to $I$.  Possibly
blowing up more we may assume that $\pi$ is a sequence of smooth blow ups.  If $Z\subset
X$ is smooth of codimension $k$ and $\llist a.k.$ are the coefficients of the components
of $\Delta$ containing $Z$ then the coefficient of the exceptional divisor is 
\[
a=1+\sum_{i\leq k}a_i-k.
\]
If $a>0$ then $a\in I$ and we are done by induction on the number of blow ups.  
\end{proof}

\begin{lemma}\label{l_up} Let $(X,\Delta)$ be a log smooth pair where the coefficients 
of $\Delta$ belong to $(0,1]$.  Suppose that there is a projective morphism
$\psi\colon\map X.U.$, where $U$ is an affine variety.

If $(X,\Delta)$ has a weak log canonical model then there is a sequence $\pi\colon\map
Y.X.$ of smooth blow ups of the strata of $\Delta$ such that if we write
\[
K_Y+\Gamma=\pi^*(K_X+\Delta)+E,
\]
where $\Gamma\geq 0$ and $E\geq 0$ have no common components, $\pi_*\Gamma=\Delta$ and
$\pi_*E=0$ and if we write 
\[
\Gamma'=\Gamma-\Gamma\wedge N_{\sigma}(Y,K_Y+\Gamma),
\]
then $\mathbf{B}_-(Y,K_Y+\Gamma')$ contains no strata of $\Gamma'$.  If $\Delta$ is
a $\mathbb{Q}$-divisor then $\Gamma'$ is a $\mathbb{Q}$-divisor.
\end{lemma}
\begin{proof} Let $f\colon\rmap X.W.$ be a weak log canonical model of $(X,\Delta)$.  Let
$\Phi=f_*\Delta$.  Let $I$ be the finite set whose existence is guaranteed by
\eqref{l_coefficients} applied to $(W,\Phi)$.

Suppose that $\pi\colon\map Y.X.$ is a sequence of smooth blow ups of the strata of
$\Delta$.  We may write
\[
K_Y+\Gamma=\pi^*(K_X+\Delta)+E,
\]
where $\Gamma\geq 0$ and $E\geq 0$ have no common components, $\pi_*\Gamma=\Delta$ and
$\pi_*E=0$.  

Note that if $g\colon\rmap Y.W.$ is the induced birational map then $g$ is a weak log
canonical model of $(Y,\Gamma)$.  In particular if we write
\[
K_Y+\Gamma=g^*(K_W+\Phi)+E_1
\]
then $E_1=N_{\sigma}(Y,K_Y+\Gamma)$.  Thus if we write
\[
K_Y+\Gamma_0=g^*(K_W+\Phi)+E_0,
\] 
where $\Gamma_0$ and $E_0\geq 0$ have no common components, $g_*\Gamma_0=\Phi$ and $g_*E_0=0$
then 
\[
\Gamma_0=\Gamma'=\Gamma-\Gamma\wedge N_{\sigma}(Y,K_Y+\Gamma).
\]

Let $p\colon\map V.Y.$ and $q\colon\map V.W.$ resolve $g$, so that the strict transform of
$\Phi$ and the exceptional locus of $q$ has global normal crossings.  We may write
\[
K_V+\Psi=q^*(K_W+\Phi)+F,
\]
where $\Psi\geq 0$ and $F\geq 0$ have no common components, $q_*\Psi=\Phi$ and $q_*F=0$.
Note that the coefficients of $\Psi$ belong to $I$.

As $q^*(K_W+\Phi)$ is nef, $\Psi$ has no components in common with $N_{\sigma}(V,K_V+\Psi)=F$.  
As
\begin{align*} 
K_Y+p_*\Psi  &=g^*(K_W+\Phi)+p_*F\\
K_Y+\Gamma_0 &=g^*(K_W+\Phi)+E_0,
\end{align*} 
we have
\[
\Gamma_0+p_*F=p_*\Psi+E_0.
\]
As $\Gamma_0$ and $E_0$, $p_*\Psi$ and $p_*F$ have no common components, it follows that
$\Gamma'=p_*\Psi$, so that the coefficients of $\Gamma'$ belong to $I$.

Suppose that $Z$ is a strata of $(X,\Delta)$ which is contained in
$N_{\sigma}(X,K_X+\Delta)$.  Let $\pi\colon\map Y.X.$ blow up $Z$ and let $E$ be the
exceptional divisor.  The coefficient of $E$ in $\Gamma$ is no more than the minimum
coefficient of any component of $\Delta$ containing $Z$.  Either the coefficient of $E$ in
$\Gamma'$ is zero or $E$ is a component of $\Gamma-\Gamma'$, so that, either way, the
coefficient of $E$ in $\Gamma'$ is strictly less than the coefficient of any component of
$\Delta$ containing $Z$.  Since $I$ is a finite set and $(X,\Delta)$ has only finitely
many strata, it is clear that after finitely many blow ups we must have that no strata of
$(Y,\Gamma')$ is contained in $N_{\sigma}(Y,K_{Y}+\Gamma')$.
\end{proof}

\begin{lemma}\label{l_dlt} Let $(X,\Delta)$ be a log pair and let $\pi\colon\map X.U.$ be
a morphism of quasi-projective varieties.  Suppose that $U$ is smooth, $\pi$ is flat and
the support of $\Delta$ contains neither a component of a fibre nor a codimension one
singular point of any fibre.

Then the subset $U_0\subset U$ of points $u\in U$ such that the fibre $(X_u,\Delta_u)$ is
divisorially log terminal is constructible.  Further, if $U_0$ is dense in $U$ then we may
find a smooth dense open subset $U_1$ of $U$, contained in $U_0$, such that the
restriction of $(X,\Delta)$ to $U_1$ is divisorially log terminal.
\end{lemma}
\begin{proof} Let $V$ be a smooth open subset of the closure of $U_0$.  We may assume that
$V$ is irreducible.  Replacing $U$ by $V$ we may assume that $U_0$ is dense in $V$.

Let $f\colon\map Y.X.$ be a log resolution.  We may write
\[
K_Y+\Gamma=f^*(K_X+\Delta)+E,
\]
where $\Gamma\geq 0$ and $E\geq 0$ have no common components.  Passing to an open subset
of $U$ we may assume that $(Y,\Gamma)$ is log smooth over $U$.  As $\Gamma_u$ is a
boundary for a dense set of points $u\in U_0$, it follows that $\Gamma$ is a boundary.  

Suppose that $F$ is an exceptional divisor of log discrepancy zero with respect to
$(X,\Delta)$, that is, coefficient one in $\Gamma$.  Let $Z=f(F)$ be the centre of $F$ in
$X$.  Note that $F_u$ has log discrepancy zero with respect to $(X_u,\Delta_u)$, for any
$u\in U_0$.  As $(X_u,\Delta_u)$ is divisorially log terminal, it follows that
$(X_u,\Delta_u)$ is log smooth in a neighbourhood of the generic point of $Z_u$.  But then
$(X,\Delta)$ is log smooth in a neighbourhood of the generic point of $Z$ and so
$(X,\Delta)$ is divisorially log terminal.  This is the second statement.

As $(X,\Delta)$ is log smooth in a neighbourhood of the generic point of $Z$, we may find
an open subset $U_2\subset U_0$ such that if $u\in U_2$ then the $(X_u,\Delta_u)$ is log
smooth in a neighbourhood of the generic point of $Z_u$.  Possibly shrinking $U_2$ we may
also assume that the non kawamata log terminal locus of $(X_u,\Delta_u)$ is the
restriction of the non kawamata log terminal locus of $(X,\Delta)$.  It follows that if
$u\in U_2$ then $(X_u,\Delta_u)$ is divisorially log terminal.
\end{proof}

\subsection{Good minimal models}

\begin{lemma}\label{l_equivalent} Let $(X,\Delta)$ be a divisorially log terminal pair,
where $X$ is projective and $\mathbb{Q}$-factorial.  

If $(X,\Delta)$ has a weak log canonical model then the following are equivalent
\begin{enumerate} 
\item every weak log canonical model of $(X,\Delta)$ is a semi-ample model,
\item $(X,\Delta)$ has a semi-ample model, and
\item $(X,\Delta)$ has a good minimal model. 
\end{enumerate} 
\end{lemma}
\begin{proof} (1) implies (2) is clear. 

We show that (2) implies (3).  Suppose that $g\colon\rmap X.Z.$ is a semi-ample model of
$(X,\Delta)$.  Let $p\colon\map W.X.$ be a log resolution of $(X,\Delta)$ which also
resolves $g$ so that the induced rational map is a morphism $q\colon\map W.Z.$.  We may
write
\[
K_W+\Phi=p^*(K_X+\Delta)+E,
\]
where $\Phi\geq 0$ and $E\geq 0$ have no common components, $p_*\Phi=\Delta$ and $p_*E=0$.
\cite[2.10]{HX11} implies that $(X,\Delta)$ has a good minimal model if and only if
$(W,\Phi)$ has a good minimal model.

Replacing $(X,\Delta)$ by $(W,\Phi)$ we may assume that $g$ is a morphism.  We run
$f\colon\rmap X.Y.$ the $(K_X+\Delta)$-MMP with scaling of an ample divisor over $Z$.
Note that running the $(K_X+\Delta)$-MMP over $Z$ is the same as running the absolute
$(K_X+\Delta+H)$-MMP, where $H$ is the pullback of a sufficiently ample divisor from $Z$.
Note also that $N_{\sigma}(X,K_X+\Delta)$ and $N_{\sigma}(X,K_X+\Delta+H)$ have the same
components.  By (2) of \eqref{l_rational} we may run the $(K_X+\Delta)$-MMP with scaling
over $Z$ until $f$ contracts every component of $N_{\sigma}(X,K_X+\Delta)$.  If
$\Gamma=f_*\Delta$ and $h\colon\map Y.Z.$ is the induced birational morphism then $h$ only
contracts divisor on which $K_Y+\Gamma$ is trivial.  As $h_*(K_Y+\Gamma)=g_*(K_X+\Delta)$
is semi-ample it follows that
\[
K_Y+\Gamma=h^*h_*(K_Y+\Gamma),
\]
is semi-ample and so $f$ is a good minimal model.  Thus (2) implies (3).

Suppose that $f\colon\rmap X.Y.$ is a minimal model and $g\colon\rmap X.Z.$ is a weak log
canonical model.  Let $p\colon\map W.Y.$ and $q\colon\map W.Z.$ be a common resolution
over $X$, $r\colon\map W.X.$.  Then we may write
\[
p^*(K_Y+\Gamma)+E_1=r^*(K_X+\Delta)=q^*(K_Z+\Phi)+E_2
\]
where $\Gamma=f_*\Delta$, $\Phi=g_*\Delta$, $E_1\geq 0$ is $p$-exceptional and $E_2\geq 0$
is $q$-exceptional.  As $f$ is a minimal model and $g$ is a weak log canonical model,
every $f$-exceptional divisor is $g$-exceptional.  Thus
\[
p^*(K_Y+\Gamma)+E=q^*(K_Z+\Phi),
\]
where $E=E_1-E_2$ is $q$-exceptional.  Negativity of contraction applied to $q$ implies
that $E\geq 0$, so that $E\geq 0$ is $p$-exceptional.  Negativity of contraction applied
to $p$ implies that $E=0$.  But then $K_Y+\Gamma$ is semi-ample if and only if $K_Z+\Phi$ is 
semi-ample.  Thus (3) implies (1).  
\end{proof}

\begin{lemma}\label{l_line} Let $(X,\Delta)$ be a divisorially log terminal pair, where 
$X$ is projective.  Let $A$ be an ample dvisor.  Let $\pi\colon\map V.X.$ be a
divisorially log terminal modification of $X$, so that $\pi$ is small and if we write
\[
K_V+\Sigma=\pi^*(K_X+\Delta),
\]
then $(V,\Sigma)$ is divisorially log terminal and $V$ is $\mathbb{Q}$-factorial.

If $(V,\Sigma)$ has a good minimal model then there is a constant $\epsilon>0$ with the
following properties:
\begin{enumerate} 
\item If $g_t\colon\rmap X.Z_t.$ is the log canonical model of $(X,\Delta+tA)$ then $Z_t$
is independent of $t\in (0,\epsilon)$ and there is a morphism $\map Z_t.Z_0.$.
\item If $h\colon\rmap X.Y.$ is a weak log canonical model of $(X,\Delta+tA)$ for some
$t\in [0,\epsilon)$ then $h$ is a semi-ample model of $(X,\Delta)$.
\end{enumerate} 
\end{lemma}
\begin{proof} Note that as $A$ is ample, $(X,\Delta+tA)$ has a log canonical model $Z_t$
for $t>0$ by \cite[1.1]{BCHM10}.  Note also that since $\pi$ is small, $V$ and $X$ have
the same log canonical models and weak log canonical models.  At the expense of dropping
the hypothesis that $A$ is ample, replacing $X$ by $V$ we may assume that $X$ is
$\mathbb{Q}$-factorial.

Suppose that we run $f_t\colon\rmap X.W_t.$ the $(K_X+\Delta)$-MMP with scaling of $A$.
\cite[1.9.iii]{Birkar12b} implies that this MMP terminates with a minimal model, so that
we may find $\epsilon>0$ such that $f=f_0=f_t\colon\rmap X.W=W_t.$ is independent of
$t\in [0,\epsilon)$.  Let $\Phi=f_*\Delta$ and let $B=f_*A$.  If $C\subset W$ is a curve
such that $(K_W+\Phi+sB)\cdot C=0$ for some $s\in (0,\epsilon)$, then
\[
(K_W+\Phi+tB)\cdot C=0 \qquad \text{for all} \qquad t\in [0,\epsilon),
\]
since $K_W+\Phi+\lambda B$ is nef for all $\lambda\in (0,\epsilon)$.  The induced
contraction morphism $\map W.Z_t.$ to the ample model contracts those curves $C$ such that
$(K_W+\Phi+tB)\cdot C=0$ so that $Z=Z_t$ is independent of $t\in (0,\epsilon)$ and there
is a contraction morphism $\map Z_t.Z_0.$.  This is (1).

Let $h\colon\rmap X.Y.$ be a weak log canonical model of $(X,\Delta+tA)$.  Then $h$ is a
semi-ample model of $(X,\Delta+tA)$ and there is an induced morphism $\psi\colon \map
Y.Z.$.

Possibly replacing $\epsilon$ with a smaller number, \eqref{l_weak} implies that $h$
contracts every component of $N_{\sigma}(X,K_X+\Delta)$, independently of the choice of
weak log canonical model.  Note that if $P$ is a prime divisor which is not a component of
$N_{\sigma}(X,K_X+\Delta)$ then the restriction to $P$, of the birational map associated
to some multiple of $K_X+\Delta+tA$, is birational.  In particular \eqref{l_weak} implies
that $h$ does not contract $P$.  Thus $h$ contracts the components of
$N_{\sigma}(X,K_X+\Delta)$ and no other divisors.  Since $Z$ is a log canonical model of
$(X,\Delta +tA)$, then $\rmap X.Z.$ also contracts the components of
$N_{\sigma}(X,K_X+\Delta)$ and no other divisors.  It follows that $\psi$ is a small
morphism.

If $\Gamma=h_*\Delta$, $B=h_*A$, $\Psi=\psi_*\Gamma$ and $C=\psi_*B$ then 
\[
K_Y+\Gamma+sB=\psi^*(K_Z+\Psi+sC),
\]
for any $s$.  By assumption $K_Z+\Psi+sC$ is ample for $s\in (0,\epsilon)$ and 
so $K_Y+\Gamma+sB$ is nef for $s\in (0,\epsilon)$.  Thus $K_Y+\Gamma$ is nef and 
so $h$ is a semi-ample model of $(X,\Delta)$ by \eqref{l_equivalent}.  
\end{proof}

\begin{lemma}\label{t_non-closed} Let $k$ be any field of characteristic zero 
and let $(X,\Delta)$ be a log pair over $k$, where $X$ is a projective variety.  Let
$(\bar X,\bar \Delta)$ be the corresponding pair over the algebraic closure $\bar k$ of
$k$.  Assume that $(\bar X,\bar \Delta)$ is divisorially log terminal and $\bar X$ is
$\mathbb{Q}$-factorial.

Then $(X,\Delta)$ has a good minimal model if and only if $(\bar X,\bar \Delta)$ has a
good minimal model.
\end{lemma}
\begin{proof} If $W$ is a scheme over $k$ then $\bar W$ denotes the corresponding scheme
over $\bar k$.  If $f\colon\rmap X.Y.$ is a good minimal model of $(X,\Delta)$ then
$\bar f\colon\rmap \bar X.\bar Y.$ is a semi-ample model of $(\bar X,\bar \Delta)$ and so
$(\bar X,\bar \Delta)$ has a good minimal model by \eqref{l_equivalent}.

Conversely suppose that $(\bar X,\bar \Delta)$ has a good minimal model.  Pick an ample
divisor $A$ on $X$.  We run $f\colon\rmap X.Y.$ the $(K_X+\Delta)$-MMP with scaling of
$A$.  Then $f$ is a weak log canonical model of $(X,\Delta+tA)$ and so $\bar f\colon\rmap
\bar X.\bar Y.$ is a weak log canonical model of $(\bar X,\bar\Delta+t\bar A)$.
\eqref{l_line} implies that we may find $\epsilon>0$ such that $\bar f$ is a semi-ample
model of $(\bar X,\bar \Delta)$ for $t\in [0,\epsilon)$.  If $\Gamma=f_*\Delta$ then $K_{\bar
  Y}+\bar\Gamma$ is semi-ample so that $K_Y+\Gamma$ is semi-ample.  But then $f$ is a good
minimal model of $(X,\Delta)$.
\end{proof}

\makeatletter
\renewcommand{\thetheorem}{\thesection.\arabic{theorem}}
\@addtoreset{theorem}{section}
\makeatother
\section{The MMP in families I}

\begin{lemma}\label{l_central} Let $(X,\Delta)$ be a divisorially log terminal pair and let 
$\pi\colon\map X.U.$ be a projective contraction morphism, where $U$ is smooth, affine, of
dimension $k$ and $X$ is $\mathbb{Q}$-factorial.  Let $0\in U$ be a closed point such that
\begin{enumerate}
\item there are $k$ divisors $\llist D.k.$ containing $0$ such that if $H_i=\pi^*D_i$ and
$H=\alist H.+.k.$ is the sum then $(X,H+\Delta)$ is divisorially log terminal, 
\item $X_0$ is integral, $\dim X_0=\dim X-\dim U$ and $\dim V_0=\dim V-\dim U$,
for all non canonical centres $V$ of $(X,\Delta)$, and
\item $\mathbf{B}_-(X_0,K_{X_0}+\Delta_0)$ contains no non canonical centres of $(X_0,\Delta_0)$.
\end{enumerate}

Let $f\colon\rmap X.Y.$ be a step of the $(K_X+\Delta)$-MMP.  If $f$ is birational and $V$
is a non canonical centre of $(X,\Delta)$ then $f$ is an isomorphism in a neighborhood of
the generic point of $V$ and $f_0$ is an isomorphism in a neighborhood of the generic
point of $V_0$.  In particular the induced maps $\phi\colon\rmap V.W.$ and
$\phi_0\colon\rmap V_0.W_0.$ are birational, where $W=f(V)$.  Let $\Gamma=f_*\Delta$.
Further
\begin{enumerate}
\item if $G_i$ is the pullback of $D_i$ to $Y$ and $G=\alist G.+.k.$ is the sum then
$(Y,G+\Gamma)$ is divisorially log terminal, 
\item $Y_0$ is integral, $\dim Y_0=\dim Y-\dim U$ and $\dim W_0=\dim W-\dim U$,
for all non canonical centres $W$ of $(Y,\Gamma)$, and
\item $\mathbf{B}_-(Y_0,K_{Y_0}+\Gamma_0)$ contains no non canonical centres of $(Y_0,\Gamma_0)$.  
\end{enumerate}
If $V$ is a non kawamata log terminal centre, or $V=X$ then $\phi\colon\rmap V.W.$ and
$\phi_0\colon\rmap V_0.W_0.$ are birational contractions.

On the other hand, if $f$ is a Mori fibre space then $f_0$ is not birational.
\end{lemma}
\begin{proof} Suppose that $f$ is birational.  

As $f$ is a step of the $(K_X+\Delta)$-MMP and $H$ is pulled back from $U$, it follows
that it is also a step of the $(K_X+H+\Delta)$-MMP, and so $(Y,G+\Gamma)$ is divisorially
log terminal.  As every component of $Y_0$ is a non kawamata log terminal centre of
$(Y,G)$ and $X_0$ is irreducible, it follows that $Y_0$ is irreducible.

Let $V$ be a non canonical centre of $(X,\Delta)$.  Then $V$ is a non canonical centre of
$(X,H+\Delta)$.  Let $g\colon\map X.Z.$ be the contraction of the extremal ray associated
to $f$ (so that $f=g$ unless $f$ is a flip).  Every component of $V_0$ is a non-canonical
centre of $(X_0,\Delta_0)$ \cite[1.4.5]{BCHM10} and so no component of $V_0$ is contained
in $\mathbf{B}_-(X_0,K_{X_0}+\Delta_0)$ by hypothesis.  On the other hand, note that the
locus where $g$ is not an isomorphism is the locus of curves $C$ such that
$(K_X+H+\Delta)\cdot C<0$.  Thus the locus where $g_0$ is not an isomorphism is equal to
the locus of curves $C_0\subset X_0$ such that $(K_{X_0}+\Delta_0)\cdot C_0<0$.  As every
such curve $C_0$ is contained in $\mathbf{B}_-(X_0,K_{X_0}+\Delta_0)$ it follows that the
locus where $g_0$ (respectively $g$) is not an isomorphism intersects $V_0$ (respectively
$V$) in a proper closed subset.  In particular both $\phi\colon\rmap V.W.$ and
$\phi_0\colon\rmap V_0.W_0.$ are birational.

Now suppose that $V$ is a non kawamata log terminal centre or $V=X$.  If $V$ is a non
kawamata log terminal centre then $V$ is a non canonical centre and so $\phi\colon\rmap
V.W.$ and $\phi_0\colon\rmap V_0.W_0.$ are both birational.  We can define divisors
$\Sigma_0$ and $\Theta_0$ on $V_0$ and $W_0$ by adjunction:
\[
(K_{X_0}+\Delta_0)|_{V_0}=K_{V_0}+\Sigma_0.\qquad \text{and} \qquad (K_{Y_0}+\Gamma_0)|_{W_0}=K_{W_0}+\Theta_0.
\]
If $P$ is a divisor on $W_0$ and $f$ is not an isomorphism at the generic point of the
centre $N$ of $P$ on $V_0$ then
\[
a(P; V_0,\Sigma_0) < a(P; W_0,\Theta_0)\leq 1.
\] 
Thus $N$ is a non-canonical centre of $(X,\Delta)$.  Therefore $N$ is birational to $P$ so
that $N$ is a divisor on $V_0$.  Thus $\phi_0\colon\rmap V_0.W_0.$ is a birational
contraction.  In particular $f_0\colon\rmap X_0.Y_0.$ is a birational contraction and so
(1--3) clearly hold.  As $\phi_0\colon\rmap V_0.W_0.$ is a birational contraction it
follows that $\phi\colon\rmap V.W.$ is a birational contraction in a neighborhood of $V_0$.

Suppose that $f$ is a Mori fibre space.  As the dimension of the fibres of $f\colon\map
X.Y.$ are upper-semicontinuous, $f_0$ is not birational.
\end{proof}

\begin{lemma}\label{l_generic-minimal}  Let $(X,\Delta)$ be a divisorially log terminal pair
and let $\pi\colon\map X.U.$ be a projective morphism, where $U$ is smooth and affine and
$X$ is $\mathbb{Q}$-factorial.  Let $\eta \in U$ be the generic point and let $0\in U$ be
a closed point.  Suppose that either 
\begin{enumerate}
\item there are $k$ divisors $\llist D.k.$ containing $0$ such that if $H_i=\pi^*D_i$ and
$H=\alist H.+.k.$ is the sum then $(X,H+\Delta)$ is divisorially log terminal, 
\item $X_0$ is integral, $\dim X_0=\dim X-\dim U$ and $\dim V_0=\dim V-\dim U$, for all
non canonical centres $V$ of $(X,\Delta)$, and 
\item $\mathbf{B}_-(X_0,K_{X_0}+\Delta_0)$ contains no non canonical centres of $(X_0,\Delta_0)$.
\end{enumerate}
or $(X,\Delta)$ is log smooth over $U$ and (3) holds.  

If $(X_0,\Delta_0)$ has a good minimal model then we may run $f\colon\rmap X.Y.$ the
$(K_X+\Delta)$-MMP until $f_{\eta}\colon\rmap X_{\eta}.Y_{\eta}.$ is a
$({X_{\eta}},\Delta _\eta)$-minimal model and $f_0\colon\rmap X_0.Y_0.$ is a semi-ample
model of $(X_0,\Delta_0)$.  If $D$ is a component of $\rfdown\Delta.$, $E$ is the image of
$D$ and $\phi\colon\rmap D.E.$ is the restriction of $f$ to $D$ then the induced map
$\phi_0\colon\rmap D_0.E_0.$ is a semi-ample model of $(D_0,\Sigma_0)$, where $\Sigma_0$
is defined by adjunction
\[
(K_{X_0}+\Delta_0)|_{D_0}=K_{D_0}+\Sigma_0.
\]
Further $\mathbf{B}_-(X,K_X+\Delta)$ contains no non-canonical centres of $(X_0,\Delta_0)$.
\end{lemma}
\begin{proof} Suppose that $(X,\Delta)$ is log smooth over $U$.  If $\llist D.k.$ are $k$
general divisors containing $0$ then $(X,H+\Delta)$ is log smooth, so that (1) and (2)
hold.  Thus we may assume (1--3) hold.

We run $f\colon\rmap X.Y.$ the $(K_X+\Delta)$-MMP with scaling of an ample
divisor $A$.  Let $\Gamma=f_*\Delta$ and $B=f_*A$.  By construction $K_Y+tB+\Gamma$ is nef
for some $t>0$.  Since $\pi\colon\map X.U.$ satisfies the hypotheses of \eqref{l_central},
$f_0\colon\rmap X_0.Y_0.$ is a weak log canonical model of $(X_0,tA_0+\Delta_0)$.

If $K_X+\Delta$ is not pseudo-effective then this MMP ends with a Mori fibre space for
some $t>0$ and so $Y_0$ is covered by curves on which $K_{Y_0}+tB_0+\Gamma_0$ is negative
by \eqref{l_central}.  This contradicts the fact that $K_{X_0}+tA_0+\Delta_0$ is big.
Thus $K_X+\Delta$ is pseudo-effective and given any $\epsilon>0$ we may run the MMP until
$t<\epsilon$.

Since $K_{X_0}+\Delta_0$ has a good minimal model \eqref{l_line} implies that there is a
constant $\epsilon>0$ such that if $t\in (0,\epsilon)$ then any more steps of this MMP are
an isomorphism in a neighbourhood of $Y_0$.  It follows that
$K_{Y_{\eta}}+tB_{\eta}+\Gamma_{\eta}$ is nef for all $t\in (0,\epsilon)$, so that
$K_{Y_{\eta}}+\Gamma_{\eta}$ is nef.  As we are running a MMP $Y$ is
$\mathbb{Q}$-factorial and so $Y_\eta$ is $\mathbb{Q}$-factorial.  Thus
$f_{\eta}\colon\rmap X_{\eta}.Y_{\eta}.$ is a minimal model of $(X_{\eta},\Delta_{\eta})$.

Suppose that $D$ is a component of $\rfdown\Delta.$.  \eqref{l_central} implies that the
induced map $\phi_0\colon\rmap D_0.E_0.$ is a birational contraction so that $\phi_0$ is a
semi-ample model of $(D_0,\Sigma_0)$.

As 
\[
(K_Y+\Gamma)|_{Y_0}=K_{Y_0}+\Gamma_0 
\]
is nef, it follows that $\mathbf{B}_-(Y,K_Y+\Gamma)$ does not intersect $Y_0$.  Let $G$ be
an ample $\mathbb{Q}$-divisor on $Y$.  Then the stable base locus of $K_Y+\Gamma+tG$ does
not intersect $Y_0$ for any $t>0$.  If $x\in X_0$ is a point where $f$ is an isomorphism
then $x$ is not a point of the stable base locus of $K_X+\Delta+f^*(tG)$.  As $t>0$ is
arbitrary, it follows that $\mathbf{B}_-(X,K_X+\Delta)|_{X_0}$ is contained in the locus
where $f\colon\rmap X.Y.$ is not an isomorphism.  By \eqref{l_central} $f$ is an
isomorphism in a neighbourhood of any non-canonical centre.  It follows that
$\mathbf{B}_-(X,K_X+\Delta)$ contains no non-canonical centres of $(X_0,\Delta_0)$.
\end{proof}

\section{Invariance of plurigenera}

We will need the following result of B. Berndtsson and M. P\u aun.
\begin{theorem}\label{t_BP} Let $f\colon\map X.\mathbb{D}.$ be a projective contraction 
morphism to the unit disk $\mathbb{D}$ and let $(X,\Delta)$ be a log pair.

If 
\begin{enumerate}
\item $(X,\Delta)$ is log smooth over $\mathbb{D}$ and $\rfdown \Delta.=0$,
\item the components of $\Delta$ do not intersect,
\item $K_X+\Delta$ is pseudo-effective, and 
\item $\mathbf{B}_-(X,K_X+\Delta)$ does not contain any components of $\Delta_0$,
\end{enumerate} 
then 
$$
\map {H^0(X,\ring X.(m(K_X+\Delta)))}.{H^0(X_0,\ring X_0.(m(K_{X_0}+\Delta_0)))}.
$$ 
is surjective for any integer $m$ such that $m\Delta$ is integral.
\end{theorem}
\begin{proof} Note that the case $\Delta=0$ is proven in \cite{Siu02}.  Therefore we
may assume that $\Delta \neq 0$. We check that the hypotheses of \cite[Theorem 0.2]{BP10}
are satisfied and we will use the notation established there.

We take $\alpha=0$ and $p=m$ so that if $L=\ring X.(m\Delta)$ then 
$$
p([\Delta]+\alpha)=m[\Delta]\in c_1(L),
$$
is automatic.  $K_X+\Delta$ is pseudo-effective by assumption.  As we are assuming (4),
$\nu_{\min}( \{K_X+\Delta \},X_0)=0$ and $\rho ^j_{\min,\infty}=0$.  In particular $J=J'$
and $\Xi=0$.  As we are assuming that the components of $\Delta$ do not intersect the
transversality hypothesis is automatically satisfied.

If 
$$
u\in H^0(X_0,\ring X_0.(m(K_{X_0}+\Delta _0)))
$$ 
is a non-zero section then we choose $h_0=e^{-\varphi _0}$ such that $\varphi _0\leq 0=\varphi_{\Xi}$ and
$$
\Theta _{h_0}(K_{X_0}+\Delta _0)\geq 0.
$$  
Since $u$ has no poles and $\rfdown \Delta.=0$, we have
$$
\int _{X_0}\! e^{\varphi_0-\frac 1 m \varphi _{m\Delta}}<\infty.
$$ 

\cite[Theorem 0.2]{BP10} implies that  we can extend $u$ to 
\[
U\in H^0(X,\ring X.(m(K_X+\Delta))).\qedhere
\]
\end{proof}

\begin{theorem}\label{t_inv} Let $\pi\colon\map X.U.$ be a projective contraction 
morphism to a smooth variety $U$ and let $(X,\Delta)$ be a log smooth pair over $U$
such that $\rfdown\Delta.=0$.  

Then
\[
h^0(X_u,\ring X_u.(m(K_{X_u}+\Delta_u))),
\] 
is independent of the point $u\in U$, for all positive integers $m$.

In particular $\kappa(X_u,K_{X_u}+\Delta_u)$ is independent of $u\in U$ and 
\[
\map {f_*\ring X.(m(K_X+\Delta))}.{H^0(X_u,\ring X_u.(m(K_{X_u}+\Delta_u)))}.
\] 
is surjective for all positive integers $m>0$ and for all $u\in U$.
\end{theorem}
\begin{proof} Fix a positive integer $m$.  We may assume that $U$ is affine.  We may also
assume that the strata of $\Delta$ have irreducible fibers over $U$, cf. the proof of
\cite[4.2]{HMX10}.

Replacing $\Delta$ by
\[
\Delta_m=\frac{\rfdown m\Delta.}m
\]
we may assume that $m\Delta$ is integral.

By \eqref{l_klt} there is a composition of smooth blow ups of the strata of $\Delta$ such
that if we write
\[
K_Y+\Gamma=\pi^*(K_X+\Delta)+E,
\]
where $\Gamma\geq 0$ and $E\geq 0$ have no common components, $\pi_*\Gamma=\Delta$ and
$\pi_*E=0$, then no two components of $\Gamma$ intersect.  Then $(Y,\Gamma)$ is log smooth
over $U$, $m\Gamma$ is integral and $\rfdown\Gamma.=0$.

As 
\[
h^0(Y_u,\ring Y_u.(m(K_{Y_u}+\Gamma_u)))=h^0(X_u,\ring X_u.(m(K_{X_u}+\Delta_u))),
\]
replacing $(X,\Delta)$ by $(Y,\Gamma)$ we may assume that no two components of $\Delta$
intersect.  

We may assume that 
\[
h^0(X_u,\ring X_u.(m(K_{X_u}+\Delta_u)))\neq 0,
\]
for some $u\in U$.  Let $F$ be the fixed divisor of the linear system
$|m(K_{X_u}+\Delta_u)|$ and let
$$
\Theta_u=\Delta_u-\Delta _u\wedge F/m.
$$
There is a unique divisor $0\leq \Theta\leq\Delta$ such that 
\[
\Theta|_{X_u}=\Theta_u.
\]
Note that $m\Theta$ is integral, 
\begin{align*} 
f_*\ring X.(m(K_X+\Theta))&\subset f_*\ring X.(m(K_X+\Delta)) \\
\intertext{and} 
H^0(X_u,\ring X_u.(m(K_{X_u}+\Theta _u)))&=H^0(X_u,\ring X_u.(m(K_{X_u}+\Delta_u))).
\end{align*} 
Replacing $(X,\Delta)$ by $(X,\Theta)$ we may assume that no component of $\Delta_u$ is in
the base locus of $|m(K_{X_u}+\Delta_u)|$.  In particular
$\mathbf{B}_-(X_u,K_{X_u}+\Delta_u)$ does not contain any components of $\Delta_u$.  Let
$A$ be an ample divisor on $X$. We may assume that $(X,\Delta +A)$ is log smooth over $U$.
Since $K_{X_u}+\Delta _u+tA_u$ is big and $(X_u, \Delta _u+tA_u)$ is kawamata log terminal
for any $0<t<1$, it follows that $(X_u, \Delta _u+tA_u)$ has a good minimal model.
\eqref{l_generic-minimal} implies that $\mathbf{B}_-(X,K_X+\Delta+tA)$ does not contain
any components of $\Delta_u$ for any $0<t<1$.  Since
\[
\mathbf{B}_-(X,K_X+\Delta)=\bigcap _{t>0}\mathbf{B}_-(X,K_X+\Delta+tA),
\] 
it follows that $\mathbf{B}_-(X,K_X+\Delta)$ does not contain any components of $\Delta_u$
and we may apply \eqref{t_BP}.
\end{proof}

Using \eqref{t_inv} we can give another proof of \cite[(1.8)]{HMX12}:
\begin{corollary}\label{c_inv}  Let $\pi\colon\map X.U.$ be a projective contraction 
morphism to a smooth variety $U$.  

If $(X,\Delta)$ is a log smooth pair over $U$ and the coefficients of $\Delta$ are at most
one then $\vol(X_u,K_{X_u}+\Delta_u)$ is independent of $u\in U$.
\end{corollary}
\begin{proof} If $\epsilon\in (0,1]$ is a rational number then
$\rfdown(1-\epsilon)\Delta.=0$ and so \eqref{t_inv} implies that
$h^0(X_u,\ring X_u.(m(K_{X_u}+(1-\epsilon)\Delta_u)))$ is independent of the point
$u\in U$, for all sufficiently divisible integers $m>0$.  In particular
$\vol(X_u,K_{X_u}+(1-\epsilon)\Delta_u)$, is independent of the point $u\in U$.  By
continuity of the volume, it follows that $\vol(X_u,K_{X_u}+\Delta_u)$ is independent of
the point $u\in U$.
\end{proof}
\section{The MMP in families II}

\begin{lemma}\label{l_index} Let $(X,\Delta)$ be a log canonical pair and let $(X,\Phi)$ 
be a divisorially log terminal pair, where $X$ is $\mathbb{Q}$-factorial of dimension $n$.
Let
\[
\Delta(t)=(1-t)\Delta+t\Phi.
\]
Suppose that $\map X.U.$ is projective, that $U$ is smooth and affine, and that the fibres
of $\pi$ all have the same dimension.  Let $f\colon\rmap X.Y.$ be a step of the
$(K_X+\Delta(t))$-MMP over $U$ and let $\Gamma=f_*\Delta$.

Suppose $0\in U$ is a closed point such that $X_0$ is reduced, no component of $X_0$ is
contained in the support of $\Delta$, $K_{X_0}+\Delta_0$ is nef and $(X_0,\Delta_0)$
is log canonical.  Let $r$ be a positive integer, such that $r(K_{X_0}+\Delta_0)$ is
Cartier.  

If 
\[
0<t\leq \frac 1{1+2nr}
\] 
then $f$ is $(K_X+\Delta)$-trivial in a neighbourhood of $X_0$.  In particular
$(Y_0,\Gamma_0)$ is log canonical, $K_{Y_0}+\Gamma_0$ is nef, $r(K_{Y_0}+\Gamma_0)$ is
Cartier and $(Y,\Gamma)$ is log canonical in a neighbourhood of $Y_0$.
\end{lemma}
\begin{proof} Let $R$ be the extremal ray corresponding to $f$.  

If $f$ is an isomorphism in a neighbourhood of $X_0$ there is nothing to prove and if
$(K_X+\Delta)\cdot R=0$, the result follows by \cite[3.17]{KM98}.  

Otherwise, as $K_{X_0}+\Delta_0$ is nef, $(K_X+\Delta)\cdot R>0$ and so
$(K_X+\Phi)\cdot R<0$.  \cite{Kawamata91} (see also \cite[3.8.1]{BCHM10}) implies that $R$
is spanned by a rational curve $C$ contained in $X_0$ such that
\[
-(K_X+\Phi)\cdot C\leq 2n.
\]
As $r(K_{X_0}+\Delta_0)$ is Cartier
\[
(K_X+\Delta)\cdot C=(K_{X_0}+\Delta_0)\cdot C\geq \frac 1r.
\]
Thus 
\begin{align*} 
0 &>(K_X+\Delta(t))\cdot C \\
  &=(1-t)(K_X+\Delta)\cdot C+t(K_X+\Phi)\cdot C \\
  &\geq \frac{(1-t)}r-2nt\\
  &=\frac 1r-t\frac {(1+2nr)}r\\
  &\geq 0,
\end{align*} 
a contradiction.
\end{proof}

\begin{lemma}\label{l_generic-fano} Let $(X,\Delta=S+B)$ be a divisorially log terminal pair, 
where $S\leq \rfdown\Delta.$ and $X$ is $\mathbb{Q}$-factorial.  Let $\pi\colon\map X.U.$
be a projective morphism, where $U$ is smooth and affine, and the fibres of $\pi$ have the
same dimension.  Let $0\in U$ be a closed point, such that $X_0$ is integral, let $n$ be
the dimension of $X$ and let $r$ be a positive integer such that $r(K_{X_0}+\Delta_0)$ is
Cartier.  Suppose that $X_0$ is not contained in the support of $\Delta$.  Fix
\[
\epsilon<\frac 1{2nr+1}.
\]

If $(X_0,\Delta_0)$ is log canonical, $K_{X_0}+\Delta_0$ is nef but $K_X+(1-\epsilon)S+B$
is not pseudo-effective, then we may run $f\colon\rmap X.Y.$ the
$(K_X+(1-\epsilon)S+B)$-MMP over $U$, the steps of which are all $(K_X+\Delta)$-trivial in
a neighbourhood of $X_0$, until we arrive at a Mori fibre space $\psi\colon\map Y.Z.$ such
that the strict transform of $S$ dominates $Z$ and $K_Y+\Gamma\sim_{\mathbb{Q}}\psi^*L$,
for some divisor $L$ on $Z$.
\end{lemma}
\begin{proof} We run $f\colon\rmap X.Y.$ the $(K_X+(1-\epsilon)S+B)$-MMP with scaling of
an ample divisor over $U$.  \eqref{l_index} implies that every step of this MMP is
$(K_X+\Delta)$-trivial in a neighbourhood of $X_0$.  As $K_X+(1-\epsilon)S+B$ is not
pseudo-effective this MMP ends with a Mori fibre space $\psi\colon\map Y.Z.$.  As every
step of this MMP is $(K_X+\Delta)$-trivial in a neighbourhood of $X_0$, it follows that
the strict transform of $S$ dominates $Z$.
\end{proof}

\begin{lemma}\label{l_base} Let $(X,\Delta)$ be a divisorially log terminal pair, 
where $X$ is $\mathbb{Q}$-factorial and projective and $\Delta$ is a $\mathbb{Q}$-divisor.

If $\Phi$ is a $\mathbb{Q}$-divisor such that 
\[
0\leq \Delta-\Phi\leq N_{\sigma}(X,K_X+\Delta),
\]
then $(X,\Phi)$ has a good minimal model if and only if $(X,\Delta)$ has a good minimal
model.
\end{lemma}
\begin{proof} Suppose that $f\colon\rmap X.Y.$ is a minimal model of $(X,\Delta)$.  Let
$\Gamma=f_*\Delta$.  (2) of \eqref{l_rational} implies that $f$ contracts every component
of $N_{\sigma}(X,K_X+\Delta)$ so that
\[
f_*(K_X+\Delta)=K_Y+\Gamma=f_*(K_X+\Phi).
\]
Let $p\colon\map W.X.$ and $q\colon\map W.Y.$ resolve $f$.  If we write
\[
p^*(K_X+\Delta)=q^*(K_Y+\Gamma)+E,
\]
then $E\geq 0$ is $q$-exceptional and $p_*E=N_{\sigma}(X,K_X+\Delta)$.  It follows that 
if we write 
\[
p^*(K_X+\Phi)=q^*(K_Y+\Gamma)+F,
\]
then 
\[
F=E-p^*(\Delta-\Phi)\geq E-p^*(N_{\sigma}(X,K_X+\Delta))=E-p^*p_*E.
\]
As $E-p^*p_*E$ is $p$-exceptional, $p_*F\geq 0$ by the negativity lemma and so $f$ is a
weak log canonical model of $(X,\Phi)$.  If $f$ is a good minimal model of $(X,\Delta)$
then $f$ is a semi-ample model of $(X,\Phi)$ and so $(X,\Phi)$ has a good minimal model by
\eqref{l_equivalent}.

Now suppose that $(X,\Phi)$ has a good minimal model.  We may run the $(K_X+\Phi)$-MMP
until we get a minimal model $f\colon\rmap X.Y.$ of $(X,\Phi)$.  Let $\map Y.Z.$ be
the ample model of $K_X+\Phi$.

If $t>0$ is sufficiently small then $f$ is also a run of the $(K_X+\Delta_t)$-MMP, where
\[
\Delta_t=\Phi+t(\Delta-\Phi).
\]  

Let $n$ be the dimension of $X$ and let $r$ be a positive integer such that $r(K_X+\Phi)$ 
is Cartier.  If 
\[
0<t<\frac 1{1+2nr}
\]
and we continue to run the $(K_X+\Delta_t)$-MMP with scaling of an ample divisor then
\eqref{l_index} (taking $U$ to be a point) implies that every step of this MMP is
$(K_X+\Phi)$-trivial, so that every step is over $Z$.  After finitely many steps
\eqref{l_rational} implies that we obtain a model $g\colon\rmap X.W.$ which contracts the
components of $N_{\sigma}(X, K_X+\Delta_t)$.  As the support of
$N_{\sigma}(X, K_X+\Delta)$ is the same as the support of $N_{\sigma}(X, K_X+\Delta_t)$
and the support of $\Delta-\Phi$ is contained in $N_{\sigma}(X, K_X+\Delta)$ it follows
that
\[
g_*(K_X+\Delta)=g_*(K_X+\Phi).
\]
Thus $g_*(K_X+\Delta)$ is semi-ample.  On the other hand $g$ only contracts divisors in
$N_{\sigma}(X,K_X+\Delta)$ so that \eqref{l_only} implies that $g$ is a minimal model of
$(X,\Delta)$.  Thus $g\colon\rmap X.W.$ is a good minimal model of $(X,\Delta)$.
\end{proof}

\section{Abundance in families}
\label{generic-abundance}

\begin{lemma}\label{l_one-smooth} Suppose that $(X,\Delta)$ is a log pair where the
coefficients of $\Delta$ belong to $(0,1]\cap \mathbb{Q}$.  Let $\pi\colon\map X.U.$ be a
projective morphism to a smooth affine variety $U$.  Suppose that $(X,\Delta)$ is log
smooth over $U$.

If there is a closed point $0\in U$ such that the fibre $(X_0,\Delta_0)$ has a good
minimal model then the generic fibre $(X_{\eta},\Delta_{\eta})$ has a good minimal model.
\end{lemma}
\begin{proof} By \eqref{t_non-closed} it is enough to prove that the geometric generic
fibre has a good minimal model.  Replacing $U$ by a finite cover we may therefore assume
that $\pi$ is a contraction morphism and the strata of $\Delta$ have irreducible fibres
over $U$.

Let $f_0\colon\map Y_0.X_0.$ be the birational morphism given by \eqref{l_up}.  As
$(X,\Delta)$ is log smooth over $U$, the strata of $\Delta$ have irreducible fibres over
$U$ and $f_0$ blows up strata of $\Delta_0$, we may extend $f_0$ to a birational morphism
$f\colon\map Y.X.$ which is a composition of smooth blow ups of strata of $\Delta$.  We may
write
\[
K_Y+\Gamma=f^*(K_X+\Delta)+E,
\]
where $\Gamma\geq 0$ and $E\geq 0$ have no common components, $f_*\Gamma=\Delta$ and
$f_*E=0$.  $(Y,\Gamma)$ is log smooth and the fibres of the components of $\Gamma$ are
irreducible.  \cite[2.10]{HX11} implies that $(Y_0,\Gamma_0)$ has a good minimal model, as
$(X_0,\Delta_0)$ has a good minimal model; similiarly \cite[2.10]{HX11} also implies that
if $(Y_{\eta},\Gamma_{\eta})$ has a good minimal model then $(X_{\eta},\Delta_{\eta})$ has
a good minimal model.

Replacing $(X,\Delta)$ by $(Y,\Gamma)$ we may assume that if 
\[
\Theta_0=\Delta_0-\Delta_0\wedge N_{\sigma}(X_0,K_{X_0}+\Delta_0)
\]
then $\mathbf{B}_-(X_0,K_{X_0}+\Theta_0)$ contains no strata of $\Theta_0$.  There is a
unique divisor $0\leq \Theta\leq \Delta$ such that $\Theta|_{X_0}=\Theta_0$.
\eqref{l_commute} implies that
\[
\Theta=\Delta-\Delta\wedge N_{\sigma}(X,K_X+\Delta)
\]
so that 
\[
\Delta-\Theta\leq N_{\sigma}(X,K_X+\Delta).
\]
Hence by \eqref{l_base} and \eqref{t_non-closed} it suffices to prove that
$(X_{\eta},\Theta_{\eta})$ has a good minimal model.  Replacing $(X,\Delta)$ by
$(X,\Theta)$ we may assume that $\mathbf{B}_-(X_0,K_{X_0}+\Delta_0)$ contains no strata of
$\Delta_0$.  \eqref{l_generic-minimal} implies that we can run $f\colon\rmap X.Y.$ the
$(K_X+\Delta)$-MMP over $U$ to obtain a minimal model of the generic fibre.  Let
$\Gamma=f_*\Delta$.

Pick a component $D$ of $\rfdown\Delta.$.  Let $\phi\colon\rmap D.E.$ be the restriction
of $f$ to $D$.  \eqref{l_generic-minimal} implies that $\phi_0$ is a semi-ample model of
$(D_0,(\Delta_0-D_0)|_{D_0})$.  \eqref{l_equivalent} implies that
$(D_0,(\Delta_0-D_0)|_{D_0})$ has a good minimal model.  By induction on the dimension
$(D_{\eta},(\Delta_{\eta}-D_{\eta})|_{D_{\eta}})$ has a good minimal model.  But then
$\phi_{\eta}\colon\rmap D_{\eta}.E_{\eta}.$ is a semi-ample model of
$(D_{\eta},(\Delta_{\eta}-D_{\eta})|_{D_{\eta}})$.

Let $S=\rfdown\Delta.$ and $B=\{\Delta\}=\Delta-S$.  Let $T=f_*S$ and $C=f_*B$.  Suppose
that $K_{Y_0}+(1-\epsilon)T_0+C_0$ is not pseudo-effective for any $\epsilon>0$.  Then
$K_{X_0}+(1-\epsilon)S_0+B_0$ is not pseudo-effective for any $\epsilon>0$.  It follows easily  
that $K_X+(1-\epsilon)S+B$ is not pseudo-effective for any $\epsilon>0$.  But then
$K_Y+(1-\epsilon)T+C$ is not pseudo-effective for any $\epsilon>0$.
\eqref{l_generic-fano} implies that we may run the $(K_Y+(1-\epsilon)T+C)$-MMP until we
get to a Mori fibre space $g\colon\rmap Y.W.$, $\psi\colon\map W.V.$ over $U$.  By
assumption $g_*(K_Y+\Gamma)\sim_{\mathbb{Q}}\psi^*L$ for some divisor $L$.

Pick a component $D$ of $S$ whose image $F$ in $W$ dominates $V$.  Let $E$ be the image of
$D$ in $Y$.  As we already observed, $\phi_{\eta}\colon\rmap D_{\eta}.E_{\eta}.$ is a
semi-ample model of $(D_{\eta},(\Delta_{\eta}-D_{\eta})|_{D_{\eta}})$.  As the birational
map $g_0\colon\rmap Y_0.W_0.$ is $(K_{Y_0}+\Gamma_0)$-trivial, the birational map
$g_{\eta}\colon\rmap Y_{\eta}.W_{\eta}.$ is also $(K_{Y_{\eta}}+\Gamma_{\eta})$-trivial.
Then $L_{\eta}$ is semi-ample as $(\psi^*L)|_{F_{\eta}}$ is semi-ample.  The composition
$\rmap X_{\eta}.W_{\eta}.$ is a semi-ample model of $(X_{\eta},\Delta_{\eta})$ and so
$(X_{\eta},\Delta_{\eta})$ has a good minimal model by \eqref{l_equivalent}.

Otherwise, $K_{Y_0}+(1-\epsilon)T_0+C_0$ is pseudo-effective for some $\epsilon>0$.  If
$\map Y_0.Z_0.$ is the log canonical model of $(Y_0,\Gamma_0)$ then $T_0$ does not
dominate $Z_0$ and so if $\epsilon$ is sufficiently small then
$K_{X_0}+(1-\epsilon)S_0+B_0$ has the same Kodaira dimension as $K_{X_0}+\Delta_0$.
\begin{align*} 
\kappa(X_{\eta},K_{X_{\eta}}+\Delta_{\eta})&\geq \kappa(X_{\eta},K_{X_{\eta}}+(1-\epsilon)S_{\eta}+B_{\eta}.)\\
                                        &=\kappa(X_0,K_{X_0}+(1-\epsilon)S_0+B_0)\\
                                        &=\kappa(X_0,K_{X_0}+\Delta_0)\\
                                        &=\kappa_{\sigma}(X_0,K_{X_0}+\Delta_0)\\
                                        &=\nu(Y_0,K_{Y_0}+\Gamma_0)\\
                                        &=\nu(Y_{\eta},K_{Y_{\eta}}+\Gamma_{\eta}).
\end{align*} 

The first inequality holds as $S_{\eta}\geq 0$, the second equality holds by \eqref{t_inv}
(note that $(X_0,(1-\epsilon)S_0+B_0)$ is kawamata log terminal as $(X_0,\Delta_0)$ is
divisorially log terminal) and the last equality holds as intersection numbers are
deformation invariant.

We have already seen that if $E$ is a component of $T$ then $(K_Y+\Gamma)|_{E_{\eta}}$ is
semi-ample.  \eqref{t_slc} implies that $(K_Y+\Gamma)|_{T_{\eta}}$ is semi-ample.  Let
$H=K_{Y_{\eta}}+\Gamma_{\eta}$.  Then $H|_{T_{\eta}}$ is semi-ample and
$aH-(K_{Y_{\eta}}+\Gamma_{\eta})$ is nef and abundant for all $a>1$.  Thus
$f_{\eta}\colon\rmap X_{\eta}.Y_{\eta}.$ is a good minimal model by \eqref{t_bpf}.
\end{proof}

\begin{lemma}\label{l_global} Suppose that $(X,\Delta)$ is a log pair where the
coefficients of $\Delta$ belong to $(0,1]\cap \mathbb{Q}$.  Let $\pi\colon\map X.U.$ be a
projective morphism to a smooth affine variety $U$.  Suppose that $(X,\Delta)$ is log
smooth over $U$.

If $(X,\Delta)$ has a good minimal model then every fibre $(X_u,\Delta_u)$ has a good minimal 
model.   
\end{lemma}
\begin{proof} Replacing $U$ by a finite cover we may assume that $\pi$ is a contraction
morphism and the strata of $\Delta$ have irreducible fibres over $U$.

Let $f\colon\map Y.X.$ be the birational morphism given by \eqref{l_up}.  We may write
\[
K_Y+\Gamma=f^*(K_X+\Delta)+E,
\]
where $\Gamma\geq 0$ and $E\geq 0$ have no common components, $f_*\Gamma=\Delta$ and
$f_*E=0$.  $(Y,\Gamma)$ is log smooth.  \cite[2.10]{HX11} implies that $(Y,\Gamma)$ has a
good minimal model, as $(X,\Delta)$ has a good minimal model; similiarly \cite[2.10]{HX11}
also implies that if $(Y_u,\Gamma_u)$ has a good minimal model then
$(X_u,\Delta_u)$ has a good minimal model.

Replacing $(X,\Delta)$ by $(Y,\Gamma)$ we may assume that if 
\[
\Theta=\Delta-\Delta\wedge N_{\sigma}(X,K_X+\Delta)
\]
then $\mathbf{B}_-(X,K_X+\Theta)$ contains no strata of $\Theta$.  As 
\[
\Delta-\Theta\leq N_{\sigma}(X,K_X+\Delta)
\]
\eqref{l_base} implies that $(X,\Theta)$ has a good minimal model.  \eqref{l_commute}
implies that
\[
\Theta_u=\Delta_u-\Delta_u\wedge N_{\sigma}(X_u,K_{X_u}+\Delta_u)
\]
so that $\mathbf{B}_-(X_u,K_{X_u}+\Theta_u)$ contains no strata of $\Theta_u$.  Hence
\[
\Delta_u-\Theta_u\leq N_{\sigma}(X_u,K_{X_u}+\Delta_u).
\]
Hence by \eqref{l_base} it suffices to prove that $(X_u,\Theta_u)$ has a good minimal
model.  Replacing $(X,\Delta)$ by $(X,\Theta)$ we may assume that
$\mathbf{B}_-(X_u,K_{X_u}+\Delta_u)$ contains no strata of $\Delta_u$.

Let $A$ be an ample divisor over $U$.  \cite[2.7]{HX11} implies that the
$(K_X+\Delta)$-MMP with scaling of $A$ terminates $\pi\colon\rmap X.Y.$ with a good
minimal model for $(X,\Delta)$ over $U$.  Since $\mathbf{B}_-(X_u,K_{X_u}+\Delta_u)$
contains no strata of $\Delta_u$, \eqref{l_central} implies that $\pi_u\colon\rmap
X_u.Y_u.$ is a semi-ample model of $(X_u,\Delta_u)$.  \eqref{l_equivalent} implies that 
$(X_u,\Delta_u)$ has a good minimal model. 
\end{proof}

\begin{proof}[Proof of \eqref{t_one-smooth}] By \eqref{l_one-smooth} the generic fibre
$(X_{\eta},\Delta_{\eta})$ has a good minimal model.  Hence we may find a good minimal
model of $\pi^{-1}(U_0)$ over an open subset $U_0$ of $U$.  As $(X,\Delta)$ is log smooth
over $U$, every strata of $S=\rfdown\Delta.$ intersects $\pi^{-1}(U_0)$.  Therefore we may
apply \cite[1.1]{HX11} to conclude that $(X,\Delta)$ has a good minimal model over $U$.
\eqref{l_global} implies that every fibre has a good minimal model.
\end{proof}

\begin{proof}[Proof of \eqref{c_dense}] By \eqref{l_dlt} we may assume that $(X,\Delta)$
is divisorially log terminal and every fibre $(X_u,\Delta_u)$ is divisorially log
terminal.

It suffices to prove that if $U_0$ is dense then it contains an open subset.  Let
$\pi\colon\map Y.X.$ be a log resolution.  We may write
\[
K_Y+\Gamma=\pi^*(K_X+\Delta)+E,
\]
where $\Gamma\geq 0$ and $E\geq 0$ have no common components.  Passing to an open subset
we may assume that $(Y,\Gamma)$ is log smooth over $U$, so that 
\[
K_{Y_u}+\Gamma_u=\pi^*(K_{X_u}+\Delta_u)+E_u,
\]
for all $u\in U$.  \cite[2.10]{HX11} implies that if $(Y,\Gamma)$ has a good minimal model
over $U$ then $(X,\Delta)$ has a good minimal model over $U$.  Similarly \cite[2.10]{HX11}
implies that if $(X_u,\Delta_u)$ has a good minimal model then $(Y_u,\Gamma_u)$ has a good
minimal model.

Replacing $(X,\Delta)$ by $(Y,\Gamma)$ we may assume that $(X,\Delta)$ is log smooth over
$U$.  \eqref{t_one-smooth} implies that $U_0=U$.
\end{proof}

\begin{lemma}\label{l_deformation} Let $\pi\colon\map X.U.$ be a projective morphism 
to a smooth variety $U$ and let $(X,\Delta)$ be log smooth over $U$.  Suppose that the
coefficients of $\Delta$ belong to $(0,1]\cap \mathbb{Q}$.

If there is a closed point $0\in U$ such that the fibre $(X_0,\Delta_0)$ has a good
minimal model then the restriction morphism
\[
\map {\pi_*\ring X.(m(K_X+\Delta))}.{H^0(X_0,\ring X_0.(m(K_{X_0}+\Delta_0)))}.
\]
is surjective for any $m\in\mathbb{N}$ such that $m\Delta$ is integral.
\end{lemma}
\begin{proof} \eqref{l_one} implies that we may assume that $m\geq 2$.  Replacing $U$ by a
finite cover we may assume that $\pi$ is a contraction morphism and the strata of $\Delta$
have irreducible fibres over $U$.  Since the result is local we may assume that $U$ is
affine and so we want to show that the restriction map
\[
\map {H^0(X,\ring X.(m(K_X+\Delta)))}.{H^0(X_0, \ring X_0.(m(K_{X_0}+\Delta_0)))}.
\]
is surjective.  Cutting by hyperplanes we may assume that $U$ is a curve.  
Let $f_0\colon\map Y_0.X_0.$ be the birational morphism given by \eqref{l_up}.  As
$(X,\Delta)$ is log smooth over $U$, the strata of $\Delta$ have irreducible fibres over
$U$ and $f_0$ blows up strata of $\Delta_0$, we may extend $f_0$ to a birational morphism
$f\colon\map Y.X.$ which is a composition of smooth blow ups of strata of $\Delta$.  We may
write
\[
K_Y+\Gamma=f^*(K_X+\Delta)+E,
\]
where $\Gamma\geq 0$ and $E\geq 0$ have no common components, $f_*\Gamma=\Delta$ and
$f_*E=0$.  $(Y,\Gamma)$ is log smooth and the fibres of the components of $\Gamma$ are
irreducible.  Note that $m\Gamma$ is integral and the natural maps induce isomorphisms
\begin{align*} 
H^0(X,\ring X.(m(K_X+\Delta)))&\simeq H^0(Y,\ring Y.(m(K_Y+\Gamma)))\\
\intertext{and}
H^0(X_0, \ring X_0.(m(K_{X_0}+\Delta_0)))&\simeq H^0(Y_0, \ring Y_0.(m(K_{Y_0}+\Gamma_0)))
\end{align*}

Replacing $(X,\Delta)$ by $(Y,\Gamma)$ we may assume that if 
\[
\Theta_0=\Delta_0-\Delta_0\wedge N_{\sigma}(X_0,K_{X_0}+\Delta_0),
\]
then $\mathbf{B}_-(X_0,K_{X_0}+\Theta_0)$ contains no strata of $\Theta_0$.  There is a
unique divisor $0\leq \Theta\leq \Delta$ such that $\Theta|_{X_0}=\Theta_0$.
\eqref{t_one-smooth} implies that $K_X+\Delta$ is pseudo-effective and so
\eqref{l_commute} implies that
\[
\Theta=\Delta-\Delta\wedge N_{\sigma}(X,K_X+\Delta).
\]

As $(X_0,\Delta_0)$ has a good minimal model, \eqref{l_base} implies that $(X_0,\Theta_0)$
has a good minimal model.  Therefore \eqref{t_one-smooth} implies that $(X,\Theta)$ has a
good minimal model over $U$ and so \cite[2.9]{HX11} implies that any run of the
$(K_X+\Theta)$-MMP over $U$ with scaling of an ample divisor always terminates.
\eqref{l_generic-minimal} implies that we may run $f\colon\rmap X.Y.$ the
$(K_X+\Theta)$-MMP over $U$ until we get to a semi-ample model of the generic fibre;
\eqref{l_central} implies that $f$ is an isomorphism in a neighbourhood of the generic
point of every non kawamata log terminal centre of $(X,X_0+\Theta)$.  Since any MMP over
$U$ terminates, we may continue this MMP until we get to a good minimal over $U$, without
changing the fiber over $0$.

Let $V\subset X\times Y$ be the graph.  Then $\map V.X.$ is an isomorphism in a
neighbourhood of the generic point of each non kawamata log terminal centre of
$(X,X_0+\Theta)$.  We may find a log resolution $\map W.V.$ of the strict transform of
$\Theta$ and the exceptional divisor of $\map V.Y.$ which is an isomorphism in a
neighbourhood of the generic point of each non kawamata log terminal centre of
$(X,X_0+\Theta)$.  If $p\colon\map W.X.$ and $q\colon\map W.Y.$ are the induced morphisms
then we may write
\[
K_W+\Phi+W_0=p^*(K_X+X_0+\Theta)+E,
\]
where $W_0$ is the strict transform of $X_0$, $\Phi$ is the strict transform of
$\rfdown\Theta.$ and $\rfup E.\geq 0$ as $p$ is an isomorphism in a neighbourhood of the
generic point of each non kawamata log terminal centre of $(X,X_0+\Theta)$.

We may also write
\[
p^*((m-1)(K_X+\Theta))=q^*f_*((m-1)(K_X+\Theta))+F.
\]
Possibly shrinking $U$, we may assume $X_0$ is $\mathbb{Q}$-linearly equivalent to zero.
If we set 
\[
A=p^*(m(K_X+\Theta))+E-F, \qquad L=\rfup A. \qquad \text{and} \qquad C=\{-A\}
\]
then
\begin{align*} 
L-W_0 &=p^*(m(K_X+\Theta))+E-F+C-W_0 \\
      &=p^*(K_X+\Theta)+E+p^*((m-1)(K_X+\Theta))-F+C-W_0 \\
      &\sim_{\mathbb{Q}} K_W+\Phi+C+q^*f_*((m-1)(K_X+\Theta)).
\end{align*} 

$(W,\Phi+C)$ is log canonical, as $(W,\Phi+C)$ is log smooth and $\Phi+C$ is a boundary.
Since all non kawamata log terminal centres of $(W,\Phi+C)$ dominate $U$, a generalisation
of Koll\'ar's injectivity theorem (see \cite{Kollar86}, \cite[6.3]{Fujino09} and
\cite[5.4]{Ambro14}) implies that multiplication by a local parameter
\[
\map {H^1(W,\ring W.(L-W_0))}.{H^1(W,\ring W.(L))}.
\]
is an injective morphism and so the restriction morphism
\[
\map {H^0(W,\ring W.(L))}.{H^0(W_0,\ring W_0.(L|_{W_0}))}.
\]
is surjective.  Note that the support of $L-\rfdown q^*f_*(m(K_X+\Theta)).$ does not contain $W_0$ and
\begin{align*} 
L- \rfdown q^*f_*(m(K_X+\Theta)).&=\rfup A.-\rfdown q^*f_*(m(K_X+\Theta)).\\
                                 &\geq \rfup A-q^*f_*(m(K_X+\Theta)). \\
                                 &=\rfup E+\frac 1{m-1}F. \\
                                 &\geq 0.
\end{align*} 

We also have 
\begin{align*} 
|L|&\subset |mp^*(K_X+\Delta)+\rup E-F.| \\
   &\subset |mp^*(K_X+\Delta)+\rup E.| \\
   &=|m(K_X+\Delta)|.
\end{align*} 
Let $q_0\colon\map W_0.Y_0.$ be the restriction of $q$ to $W_0$.  We have
\begin{align*} 
|m(K_{X_0}+\Delta_0)|&=|m(K_{X_0}+\Theta_0)|  \\
                    &=|m(K_{Y_0}+f_{0*}\Theta_0)| \\
                    &=|q_0^*m(K_{Y_0}+f_{0*}\Theta_0)| \\
                    &\subset |L_{|_{W_0}}|\\
                    &=|L|_{W_0} \\
                    &\subset |m(K_X+\Delta)|_{X_0}. \qedhere
\end{align*} 
\end{proof}

\begin{proof}[Proof of \eqref{c_deformation}] Immediate from \eqref{l_deformation} and 
\eqref{t_one-smooth}.  
\end{proof}
\section{Boundedness of moduli}
\label{lc}

\begin{lemma}\label{l_point} Let $w$ be a positive real number and let $I\subset [0,1]$ 
be a set which satisfies the DCC.  Fix a log smooth pair $(Z,B)$, where $Z$ is a
projective variety.  Let $\mathfrak{F}$ be the set of all log smooth pairs $(X,\Delta)$
such that $\vol(X,K_X+\Delta)=w$, the coefficients of $\Delta$ belong to $I$ and there is
a sequence of smooth blow ups $f\colon\map X.Z.$ of the strata of $B$ such that
$f_*\Delta\leq B$.

Then there is a sequence of blow ups $\map Y.Z.$ of the strata of $B$ such that:

If $(X,\Delta)\in \mathfrak{F}$ then 
\[
\vol(Y,K_Y+\Gamma)=w
\]
where $\Gamma$ is the sum of the strict transform of $\Delta$ and the exceptional divisors
of the induced birational map $\rmap Y.X.$.
\end{lemma}
\begin{proof} Let $n=\dim Z$.  We may suppose that $1\in I$.  Let $\mathfrak{G}$ be the
set of log smooth pairs $(Y,\Gamma)$ such that $Y$ is projective of dimension $n$ and the
coefficients of $\Gamma$ belong to $I$.

As \cite[(1.3.1)]{HMX12} implies that
\[
V=\{\, \vol(Y,K_Y+\Gamma) \,|\, (Y,\Gamma)\in \mathfrak{G} \,\}
\]
satisfies the DCC we may find $\delta>0$ such that if 
\[
\vol(Y,K_Y+\Gamma)<w+\delta \qquad \text{then} \qquad \vol(Y,K_Y+\Gamma)\leq w.
\]

As the set
\[ 
\{\, \frac{r-1}r i \,|\, r\in\mathbb{N}, i\in I \,\} 
\] 
satisfies the DCC, by \cite[(1.5)]{HMX12} we may find $r\in \mathbb{N}$ such that
$K_Y+\frac{r-1}r\Gamma$ is big whenever $(Y,\Gamma)\in \mathfrak{G}$ and $K_Y+\Gamma$ 
is big.  

Pick $\epsilon>0$ such that 
\[
(1-\epsilon)^n>\frac w{w+\delta}
\] 
and set
\[
a=1-\frac{\epsilon}r.
\]
If $(Y,\Gamma)\in \mathfrak{G}$ then
\[ 
K_Y+a\Gamma=(1-\epsilon)(K_Y+\Gamma)+\epsilon\left (K_Y+\frac{r-1}r\Gamma\right ), 
\] 
so that  
\[ 
\vol(Y,K_Y+a\Gamma)\geq \vol(Y,(1-\epsilon)(K_Y+\Gamma))=(1-\epsilon)^n\vol(Y,K_Y+\Gamma). 
\] 

As $(Z,aB)$ is kawamata log terminal \eqref{l_klt} implies we may pick a birational
morphism $g\colon\map Y.Z.$ such that if we write
\[
K_Y+\Psi_0=g^*(K_Z+aB)+E_0
\]
where $\Psi_0\geq 0$ and $E_0\geq 0$ have no common components, $g_*\Psi_0=aB$ and $g_*E_0=0$, then no
two components of $\Psi_0$ intersect.  In particular $(Y,\Psi_0)$ is terminal.

Pick $(X,\Delta)\in \mathfrak{F}$ and let $\Gamma$ be the strict transform of $\Delta$
plus the exceptional divisors of the induced birational map $\rmap Y.X.$.  Let
$\Phi=g_*(a\Gamma)$.  As $\Phi\leq aB$, if we write
\[
K_Y+\Psi=g^*(K_Z+\Phi)+E
\]
where $\Psi\geq 0$ and $E\geq 0$ have no common components, $g_*\Psi=\Phi$ and
$g_*E=0$, then $\Psi\leq \Psi_0$.  In particular $(Y,\Psi)$ is terminal.

Let $\Xi=\Psi\wedge a\Gamma$ and let $\Sigma\leq \Delta$ be the strict transform of $\Xi$
on $X$.  We have
\begin{align*} 
\vol(Y,K_Y+a\Gamma) &=\vol(Y,K_Y+\Xi)\\
                    &=\vol(X,K_X+\Sigma) \\
                    &\leq \vol(X,K_X+\Delta)=w,
\end{align*}
where we used \cite[(5.3.2)]{HMX10} for the first line and we used the fact that $(Y,\Xi)$
is terminal, as $(Y,\Psi)$ is terminal, to get from the first line to the second line.

It follows that 
\[
w\leq \vol(Y,K_Y+\Gamma)\leq \frac 1{(1-\epsilon)^n}\vol(Y,K_Y+a\Gamma)<w+\delta,
\]
by our choice of $\epsilon$, so that 
\[
\vol(Y,K_Y+\Gamma)=w, 
\]
by our choice of $\delta$.  
\end{proof}

\begin{lemma}\label{l_dcc} Let $n$ be a positive integer, let $w$ be a positive real number 
and let $I\subset [0,1]$ be a set which satisfies the DCC.  Let $\mathfrak{F}$ be a set of
log canonical pairs $(X,\Delta)$ such that $X$ is projective of dimension $n$, the
coefficients of $\Delta$ belong to $I$ and $\vol(X,K_X+\Delta)=w$.

Then there is a projective morphism $\map Z.U.$ and a log smooth pair $(Z,B)$ over $U$
such that if $(X,\Delta)\in\mathfrak{F}$ then there is a point $u\in U$ and a birational
map $f_u\colon\rmap X.Z_u.$ such that
\[
\vol(Z_u,K_{Z_u}+\Phi)=w
\]
where $\Phi\leq B_u$ is the sum of the strict transform of $\Delta$ and the exceptional
divisors of $f_u^{-1}$.  
\end{lemma}
\begin{proof} We may assume that $1\in I$.  We may also assume that $\mathfrak{F}$
consists of all log canonical pairs $(X,\Delta)$ such that $X$ is projective of dimension
$n$, the coefficients of $\Delta$ belong to $I$ and $\vol(X,K_X+\Delta)=w$.

By \cite[1.3]{HMX12} there is a constant $r$ such that if $(X,\Delta)\in \mathfrak{F}$
then $\phi_{r(K_X+\Delta)}$ is birational.  (2.3.4) and (3.1) of \cite{HMX10} imply that
the set $\mathfrak{F}$ is log birationally bounded.

Therefore we may find a projective morphism $\pi\colon\map Z.U.$ and a log pair $(Z,B)$
such that if $(X,\Delta)\in \mathfrak{F}$ then there is a point $u\in U$ and a birational
map $f\colon\rmap X.Z_u.$ such that the support of the strict transform of $\Delta$ plus
the $f^{-1}$-exceptional divisors is contained in the support of $B_u$.  By standard
arguments, see for example the proof of \cite[1.9]{HMX10}, we may assume that $(Z,B)$ is
log smooth over $U$ and the intersection of strata of $B$ with the fibres is irreducible.

Let $0$ be a closed point of $U$.  Let $\mathfrak{F}_0\subset \mathfrak{F}$ be the set of
log smooth pairs $(X_0,\Delta_0)$ such that there is a sequence of smooth blow ups
$f\colon\map X_0.Z_0.$ of the strata of $B_0$ with $f_*\Delta_0\leq B_0$.  By
\eqref{l_point} there is a sequence of blow ups $g\colon\map Y_0.Z_0.$ of the strata of
$B_0$ such that if $(X_0,\Delta_0)\in\mathfrak{F}_0$ and $\Gamma_0$ is the strict transform of
$\Delta_0$ plus the exceptional divisors then
\[
\vol(Y_0,K_{Y_0}+\Gamma_0)=w.
\]
Let $g\colon\map Y.Z.$ be the sequence of blow ups of the strata of $B$ induced by $g_0$.
Replacing $(Z,B)$ by $(Y,C)$, where $C$ is the sum of the strict transform of $B$ and the
exceptional divisors of $g$, we may assume that if $(X,\Delta)\in\mathfrak{F}_0$ then
\[
\vol(Z_0,K_{Z_0}+\Psi_0)=w,
\]
where $\Psi_0=f_*\Delta\leq B_0$.  Note that by replacing $Z$ by a higher model,
$\mathfrak{F}_0$ becomes smaller.

Suppose that $(X,\Delta)\in \mathfrak{F}$.  By a standard argument, see the proof of
\cite[(1.9)]{HMX10}, we may assume that $(X,\Delta)$ is log smooth and
$f\colon\map X.Z_u.$ blows up the strata of $B_u$.  Let $h\colon\map W.Z.$ blow up the
corresponding strata of $B$ so that $W_u=X$ and $h_u=f$.  Let $\Theta$ be the divisor on
$W$ such that $\Theta_u=\Delta$ and let $f_0\colon\map W_0.Z_0.$ be the induced birational
morphism.  Then
\[
\vol(W_0,K_{W_0}+\Theta_0)=\vol(X,K_X+\Delta)=w,
\]
by deformation invariance of the volume, \eqref{c_inv}, so that $(W_0,\Theta_0)\in
\mathfrak{F}_0$.  

But then 
\[
\vol(Z_0,K_{Z_0}+\Phi_0)=w,
\]
where $\Phi_0=f_{0*}\Theta_0$.  Let $\Phi=h_*\Theta$.  Then $\Phi_u$ is the strict
transform of $\Delta$ plus the exceptional divisors and
\[
\vol(Z_u,K_{Z_u}+\Phi_u)=w,
\]
by deformation invariance of the volume, \eqref{c_inv}.
\end{proof}

\begin{proposition}\label{p_modulifinite} Fix an integer $n$, a constant $d$ and a 
set $I\subset [0,1]$ which satisfies the DCC.

Then the set $\mathfrak{F}_{\text{lc}}(n,d,I)$ of all $(X,\Delta)$ such that
\begin{enumerate} 
\item $X$ is a union of projective varieties of dimension $n$,
\item $(X,\Delta)$ is log canonical, 
\item the coefficients of $\Delta$ belong to $I$, 
\item $K_X+\Delta$ is an ample $\mathbb{Q}$-divisor, and 
\item $(K_X+\Delta)^n=d$, 
\end{enumerate} 
is bounded.  

In particular there is a finite set $I_0$ such that
$\mathfrak{F}_{\text{lc}}(n,d,I)=\mathfrak{F}_{\text{lc}}(n,d,I_0)$.
\end{proposition}
\begin{proof} If 
\[
X=\coprod_{i=1}^k X_i,
\]
and $(X_i,\Delta_i)$ is the corresponding log canonical pair then $K_{X_i}+\Delta_i$ is
ample and if $d_i=(K_{X_i}+\Delta_i)^n$ then $d=\sum d_i$.  \eqref{l_sum} and
\eqref{t_dcc} imply that there are only finitely many tuples $(\llist d.k.)$.   

Thus it is enough to show that the set $\mathfrak{F}$ of irreducible pairs $(X,\Delta)$
satisfying (1--5) is bounded.

By \eqref{l_dcc} there is a projective morphism $\map Z.U.$ and a log smooth pair $(Z,B)$
over $U$, such that if $(X,\Delta)\in\mathfrak{F}$ then there is a closed point $u\in U$
and a birational map $f_u\colon\rmap Z_u.X.$ such that 
\[
\vol(Z_u,K_{Z_u}+\Phi)=d,
\]
where $\Phi\leq B_u$ is the sum of the strict transform of $\Delta$ and the
$f_u$-exceptional divisors.  \eqref{l_vol} implies that $f_u$ is the log canonical model of
$(Z_u,\Phi)$.

On the other hand, \eqref{c_dense} implies that if we replace $U$ by a finite disjoint
union of locally closed subsets then we may assume that every fibre of $\pi$ has a log
canonical model.  Replacing $(Z,B)$ by the log canonical model over $U$, the fibres of
$\pi$ are the elements of $\mathfrak{F}$.
\end{proof}

\begin{lemma}\label{l_bounded} Let $\mathfrak{F}$ be a family of log canonical pairs
$(X,\Delta)$ where $X$ is projective, the coefficients of $\Delta$ belong to a finite set
$I$ and $K_X+\Delta$ is ample.

Let 
\[
\mathfrak{T}=\{\, (X,\Delta,\tau) \,|\, (X,\Delta)\in \mathfrak{F},\tau\colon\map S.S.\,\}
\]
where $S$ is the normalisation of a divisor supported on $\rfdown\Delta.$ and $\tau$ is an
involution which fixes the different of $(K_X+\Delta)|_S$.

If $\mathfrak{F}$ is a bounded family then so is $\mathfrak{T}$. 
\end{lemma}
\begin{proof} By assumption there is a projective morphism $\pi\colon\map Z.U.$ and a log
pair $(Z,\Sigma)$ such that if $(X,\Delta)\in \mathfrak{F}$ then there is point $u\in U$
and an isomorphism $(Z_u,\Theta)$ with $(X,\Delta)$, where $\Theta$ is a divisor supported
on $\Sigma_u$.  As $I$ is finite, possibly replacing $U$ by a disjoint union of locally
closed subsets, we may assume that $\Theta=\Sigma_u$.

Let $U_1$ be the set $u$ of points of $U$ such that $(Z_u,\Sigma_u)$ is isomorphic to some
element $(X,\Delta)$ of $\mathfrak{F}$.  Replacing $U$ by the closure of $U_1$ we may
assume that $U_1$ is dense in $U$.  In particular we may assume that $K_Z+\Sigma$ is ample
over $U$.  As the set of points where $(Z,\Sigma)$ is log canonical is constructible,
replacing $U$ by a disjoint union of finitely many locally closed subsets, we may assume
that $(Z,\Sigma)$ is log canonical; we may also assume that $\Sigma$ meets each fibre
$Z_u$ in a divisor and that $(Z_u,\Sigma_u)$ is log canonical.

Possibly replacing $U$ by finitely many disjoint copies, we may assume that there is a
divisor $C'$ on $Z$ such that if $(X,\Delta,\tau)\in \mathfrak{T}$ then $S$ corresponds to
$C'_u$.  Possibly replacing $U$ by a disjoint union of locally closed subsets we may
assume that if $C$ is the normalisation of $C'$ then $S$ is isomorphic to $C_u$.  Possibly
replacing $U$ by a disjoint union of locally closed subsets for the last time, we may
assume that if we write
\[
(K_Z+\Sigma)|_C=K_C+\Phi \qquad \text{and} \qquad (K_X+\Delta)|_S=K_S+\Theta,
\]
then $\Theta$ corresponds to $\Phi_u$.

Recall that the scheme $\Isom_U(C,C)$, which represents the functor which assigns to a
scheme $T$ over $U$ the set of all isomorphisms $\map C_T.C_T.$ over $T$, is a countable
union of quasi-projective schemes over $U$.  Pick $m$ such that $-m(K_Z+\Sigma)$ is
Cartier.  As $-m(K_Z+\Sigma)$ is ample over $U$, the subscheme of $\Isom_U(C,C)$ fixing
the line bundle $\ring C.(-m(K_Z+\Sigma))$ is a closed subscheme which is quasi-projective
over $U$.  The set of involutions fixing the different is then a closed subscheme.

It follows that $\mathfrak{T}$ is a bounded family.
\end{proof}

\begin{proof}[Proof of \eqref{t_modulifinite}] Let $\mathfrak{T}$ be the set of triples
$(X,\Delta,\tau)$ where $(X,\Delta)\in \mathfrak{F}_{\text{lc}}(n,d,I)$ and
$\tau\colon\map S.S.$ is an involution of the normalisation of a divisor supported on
$\rfdown\Delta.$, which fixes the different of $(K_X+\Delta)|_S$.

By \cite[5.13]{Kollar13}, it is enough to prove that $\mathfrak{T}$ is bounded.
\eqref{p_modulifinite} implies that $\mathfrak{F}_{\text{lc}}(n,d,I)$ is bounded and so we
may apply \eqref{l_bounded}.  \end{proof}

\bibliographystyle{hamsplain}
\bibliography{/home/mckernan/Jewel/Tex/math}

\end{document}